\documentclass[reqno,11pt]{amsart}
\usepackage{amssymb, amsthm, mathabx}
\usepackage{mathtools}
\usepackage{tikz-cd}
\usepackage[colorlinks=true,allcolors=blue,bookmarksopen=true]{hyperref}
\usepackage[all]{xy}
\usepackage[margin=1.25in]{geometry}
\usepackage{enumerate}
\usepackage{adjustbox}

\usepackage{graphbox}
\input epsf
\theoremstyle{plain}
\newtheorem{thm}{Theorem}[section]
\newtheorem{lem}[thm]{Lemma}

\newtheorem{prop}[thm]{Proposition}

\newtheorem{theoremalph}{Theorem}
\newtheorem{prob}[thm]{Problem}
\newtheorem{claim}[thm]{Claim}

\theoremstyle{definition}
\newtheorem{exl}[thm]{Example}
\newtheorem{defn}[thm]{Definition}
\newtheorem{cons}[thm]{Construction}
\newtheorem{remark}[thm]{Remark}

\newtheorem{notation}[thm]{Notation}

\DeclareMathOperator{\N}{\mathbb{N}}
\DeclareMathOperator{\Z}{\mathbb{Z}}
\DeclareMathOperator{\Q}{\mathbb{Q}}

\DeclareMathOperator{\ssm}{\smallsetminus}
\DeclareMathOperator{\Sg}{\Sigma}
\DeclareMathOperator{\GL}{GL}

\DeclareMathOperator{\Imm}{Im}
\DeclareMathOperator{\Hom}{Hom}
\DeclareMathOperator{\Span}{Span}

\DeclareMathOperator{\coker}{coker}
\DeclareMathOperator{\gr}{g-rk}
\DeclareMathOperator{\Id}{Id}

\DeclareMathOperator{\A}{\mathcal{A}}
\DeclareMathOperator{\Tor}{Tor}

\newcommand{\sm}{\setminus}
\newcommand{\ol}{\overline}
\newcommand{\wt}{\widetilde}

\title{Stabilization distance between surfaces}

\author{Allison N.\ Miller}
\address{Department of Mathematics, Rice University, Houston, TX, United States}
\email{allison.miller@rice.edu}

\author{Mark Powell}
\address{Department of Mathematical Sciences, Durham University, United Kingdom}
\email{mark.a.powell@durham.ac.uk}

\subjclass[2010]{57N13, 57N65}
\keywords{2-knots, slice discs, stabilization distance, twisted homology}

\begin{document}

\begin{abstract}
Define the \emph{1-handle stabilization distance} between two surfaces properly embedded in a fixed 4-dimensional manifold to be the minimal number of 1-handle stabilizations necessary for the surfaces to become ambiently isotopic.
For every nonnegative integer $m$ we find a pair of 2-knots in the 4-sphere whose stabilization distance equals~$m$.

Next, using a \emph{generalized stabilization distance} that counts connected sum with arbitrary 2-knots as distance zero, for every nonnegative integer $m$ we exhibit a knot $J_m$ in the 3-sphere with two slice discs in the 4-ball whose generalized stabilization distance equals~$m$. We show this using homology of cyclic covers.

Finally, we use metabelian twisted homology to show that for each~$m$ there exists a knot and pair of slice discs with generalized stabilization distance at least~$m$, with the additional property that abelian invariants associated to cyclic covering spaces coincide. This detects different choices of slicing discs corresponding to a fixed metabolising link on a Seifert surface.
\end{abstract}

\maketitle

\section{Introduction}

Given a compact, smooth, oriented 4-manifold $W$, every second homology class can be represented by some embedded surface~\cite[Prop.~1.2.3]{Gompf-stipsicz-book}.
 %Moreover, the properties of such surfaces are intimately related to the topology of the ambient 4-manifold.
%For example, the minimal genera of smoothly embedded surfaces representing a certain homology class can be used to distinguish exotic smooth structures on homeomorphic $4$-manifolds. For recent examples of this, see~\cite{Akbulut-Yasui-2010, Yasui-2015, Gompf-2017}.
A simple operation called \emph{1-handle stabilization}, illustrated in 3-dimensional space in Figure~\ref{fig:stabpic}, preserves the homology class represented by a surface while increasing the genus by one. Roughly, a 1-handle stabilization removes $D^2 \times S^0$ from $\Sigma$ and glues in $S^1 \times D^1$, with some conditions that allow this to occur ambiently in $W$ in a controlled way (see Section~\ref{subsection:stab-distances} for formal definitions).
\begin{figure}[h]
$
\begin{array}[t]{ccc}
\begin{array}{c} \includegraphics[height=1.8cm]{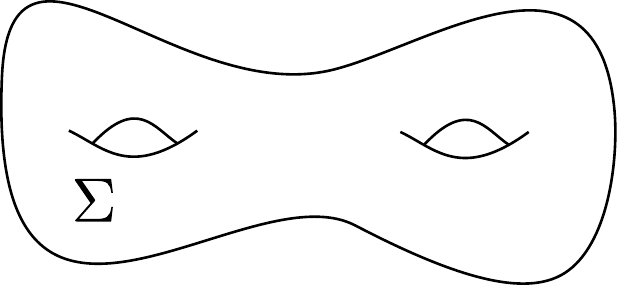} \end{array}
&
\begin{array}{c} \rightsquigarrow \end{array}
&
\begin{array}{c} \includegraphics[height=1.8cm]{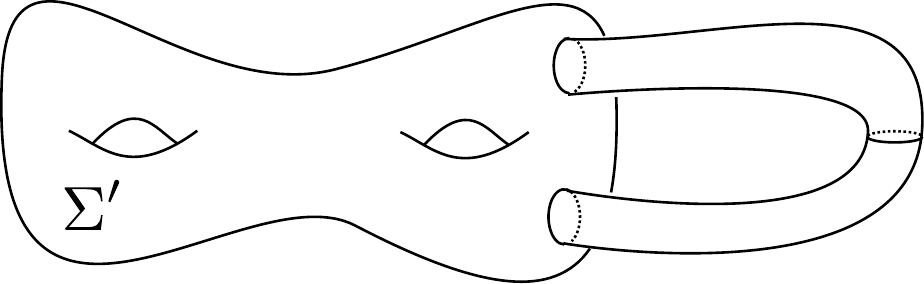} \end{array}
\end{array}$
\caption{An embedded surface $\Sigma$  (left) is stabilized by addition of a 1-handle, resulting in $\Sigma'$ (right).}\label{fig:stabpic}
\end{figure}
A result of Baykur-Sunukjian~\cite{Bay-Sun-15} states that any two embedded surfaces in $W$ representing the same second homology class become isotopic after finitely many  1-handle stabilizations.

In this paper, we  analyze the minimal number of 1-handle additions required to make two surfaces with the same genera isotopic. We call this the \emph{1-handle stabilization distance}, and show that it induces a metric on the collection of ambient isotopy classes of surfaces of a fixed genus representing a given second homology class.
There are many invariants capable of distinguishing  two surfaces up to ambient isotopy, thereby showing that at least one 1-handle addition is required, but it is more challenging to find more substantial lower bounds on the number of 1-handles needed.

Our first result shows that, even in the simplest possible setting of necessarily null-homologous 2-spheres in $S^4$, the 1-handle stabilization distance can be arbitrarily large.

\begin{theoremalph}\label{thm:A}
  For every nonnegative integer $m$, there exists a pair of embedded 2-spheres $K_1$ and $K_2$ in $S^4$ with 1-handle stabilization distance $m$.
\end{theoremalph}
We prove Theorem~\ref{thm:A} by analyzing the effect of 1-handle stabilization on the Alexander module of a surface in $S^4$.
Recall that the first Alexander module $H_1(S^{n+2} \sm \nu K;\Q[t^{\pm 1}])$ is a classical invariant of an embedded $n$-sphere $K$ in $S^{n+2}$ that measures the homology of the infinite cyclic cover of the exterior of $K$, considered as a $\Q[t^{\pm1}]$-module.
In the case of $n=1$, the order of this $\Q[t^{\pm1}]$-module is exactly the classical Alexander polynomial $\Delta_K(t)$.

\emph{Added in proof.} It was brought to our attention immediately prior to publication that Theorem~\ref{thm:A} was already proven using a similar method by Miyazaki~\cite{Miyazaki-86}.
We leave our treatment of the theorem in the paper since its primary purpose is to contrast with the upcoming theorems and their proofs, which capture more subtle phenomena.
\\

In addition to 1-handle stabilization, one might also wish to allow connected sum with arbitrary knotted 2-spheres, also called \emph{2-knots}. In the context of Theorem~\ref{thm:A} this is uninteresting: any two 2-knots become isotopic with zero 1-handle additions and a single 2-sphere addition to each. However, when considering properly embedded discs in $D^4$ with fixed boundary we show that the resulting \emph{generalized stabilization distance}, in which 1-handle addition counts as 1 and 2-sphere addition counts as 0, has similarly interesting properties.
 In particular, the generalized stabilization distance between properly embedded discs in $D^4$ with fixed boundary can be arbitrarily large.
 More precisely, a \emph{slice disc} for a 1-knot $J \subset S^3$ is a smoothly properly embedded disc $D^2 \subset D^4$ with boundary the knot $J$, and we prove the following.
\begin{theoremalph}\label{thm:B}
  For every nonnegative integer $m$, there exists a knot $J \subset S^3$ and a pair of slice discs $D_1$ and $D_2$ for $J$ with generalized stabilization distance~$m$.
\end{theoremalph}
To prove Theorem~\ref{thm:B} we again rely on the Alexander module,  comparing for $i=1$ and $2$ the kernels of the inclusion-induced maps
\[H_1(S^3 \sm \nu J;\Q[t^{\pm 1}]) \to H_1(D^4 \sm \nu D_i;\Q[t^{\pm 1}]).\]
Given any embedded surface $\Sigma$ with boundary~$J$,  we then analyze how the kernel of the inclusion induced map
\[H_1(S^3 \sm \nu J;\Q[t^{\pm 1}]) \to H_1(D^4 \sm \nu \Sigma;\Q[t^{\pm 1}])\]
can change under 1-handle and 2-sphere addition.\\

 One common way to produce a slice disc for a knot is to surger a spanning surface
for the knot along a collection of curves as follows. Given an embedded oriented surface $F$ in $S^3$ with boundary $J$, suppose we can find  a set of 0-framed curves $\gamma_i \subset F$ that form a half-basis for $H_1(F;\Z)$ and which themselves bound disjoint discs $\Delta_i$ in $D^4$. Then the surface
 \[F_{\Delta}:= \Big(F \ssm \bigcup_i (\gamma_i \times (0,1))\Big) \cup \Big( \bigcup_i \Delta_i \times \{0,1\}\Big) \subset D^4,\]
is a slice disc for $J$, after a minor isotopy to smooth corners and make the embedding proper.
  The methods of Theorem~\ref{thm:B} can often distinguish slice discs which arise from surgering a Seifert surface along two different collections of $\{\gamma_i\}$ curves.
  However,  while fixing the $\{\gamma_i\}$ there can still be multiple choices for the slice discs $\Delta_i$, and Alexander module techniques cannot distinguish the resulting slice discs for $J$.

 For our last main result we detect these second order differences between slice discs, and again show that the distance can be arbitrarily large.

\begin{theoremalph}\label{thm:C}
  For every nonnegative integer $m$, there exists a knot $J \subset S^3$ and a pair of slice discs $D_1$ and $D_2$ for $J$ with generalized stabilization distance at least $m$, such that the kernels
  \[\ker\big(H_1(S^3 \sm \nu J;\Q[t^{\pm 1}]) \to H_1(D^4 \sm \nu D_i;\Q[t^{\pm 1}])\big)\]
  coincide for $i=1,2$.
\end{theoremalph}
%\footnote{I changed $\Z$ to $\Q$ above in the interest of consistency. Feel free to change back.}
Our primary tool in the proof of Theorem~\ref{thm:C} is \emph{metabelian twisted homology}, or twisted homology coming from maps to \emph{metabelian} groups, i.e. groups $G$ with
\[ G^{(2)}:= [[G,G],[G,G]]=0.\]
These sorts of representations were notably used by Casson-Gordon \cite{Casson-Gordon:1978-1, Casson-Gordon:1986-1} to give the first examples of algebraically slice knots in $S^3$ which are not actually slice.
The corresponding twisted homology theories have the nice feature of being relatively computable while still being powerful enough to obtain strong conclusions, for example distinguishing mutant knots up to concordance \cite{Kirk-Livingston:1999-1}.
In our case, we take $G$ to be the dihedral group $D_{2n} \cong \Z_2 \ltimes \Z_n$ and construct our representations using maps from the first homology of the double cover of the relevant space to $\Z_n$.
\\

We remark that Theorem~\ref{thm:B} is not a corollary of Theorem~\ref{thm:C}, since the former gives us distance exactly $m$.
Theorem~\ref{thm:B} is also easier to prove, and the method extends straightforwardly to distinguish choices of slice discs for many knots beyond the explicit examples we give, while Theorem~\ref{thm:C} requires more involved arguments and more specialized constructions.
\\

A slightly different analysis of stabilization distance between surfaces was undertaken by \cite{Juhasz-Zemke-stab-dist}, who rather than minimizing the number of 1-handle stabilizations necessary to make two surfaces isotopic instead minimized the largest genus of any surface appearing in a sequence of stabilizations and de-stabilizations connecting the two surfaces.
\\

We also wish to advertise the following problem, which relates to recent work by \cite{Juhasz-Zemke-slice-disks} and~~\cite{Conway-Powell}.
 For a slice knot $R$, let $n_s(R)$ denote the number of equivalence classes of slice discs for $R$, where the equivalence relation is generated by connected sum with knotted 2-spheres and ambient isotopy rel.\ boundary.
Note that~$n_s(U)=1$.

Our examples of Theorem~\ref{thm:B} show that for every integer $k$ there is a knot $R_k$ with $n_s(R_k) \geq k$.
%(actually with $n_s(R_k) \geq 2^k$).  To wit,
In fact, the knot $\#^k 9_{46}$ has $2^k$ natural slice discs obtained by choosing `left band' or `right band' slice discs for each $i=1, \dots, k$; see Figure~\ref{fig:946}. By considering the kernels of the inclusion induced maps on Alexander modules as we do in the proof of Theorem~\ref{thm:B}, one can see they are all mutually not ambiently isotopic rel.\ boundary and so $n_s(\#^k 9_{46}) \geq 2^k$.

\begin{prob}
Determine the value of $n_s(R)$ for some nontrivial knot $R$, or at least whether $n_s(R)< \infty$.
\end{prob}

\subsection*{Organization of the paper}

In Section~\ref{subsection:stab-distances} we give precise definitions for our notions of stabilization distance. Section~\ref{section:cobordisms} constructs a cobordism between surface exteriors corresponding to a stabilization. Our results will follow from analyzing the effects on homology of these cobordisms. Section~\ref{section:gen-rank} recalls the notion of generating rank of a module over a commutative PID, records the facts about generating rank that we shall use, and establishes our conventions around twisted homology.  Then Section~\ref{section:2-knots} proves Theorem~\ref{thm:A}, Section~\ref{section:slice discs-1} proves Theorem~\ref{thm:B}, and Section~\ref{section:slice-discs-2} proves Theorem~\ref{thm:C}.

\subsection*{Conventions}
All manifolds, unless otherwise stated, are compact, smooth, and oriented.
When $N$ is a properly embedded submanifold of $M$, we write $X_N:= M \ssm \nu(N)$.
In our context, we will frequently have a canonical isomorphism $\varepsilon \colon H_1(X_N) \to \Z$ and  in this case we let $X_N^n$ denote the corresponding $n$-fold cyclic cover, for $n \in \mathbb{N} \cup \{ \infty\}$.
For $n \in \mathbb{N}$, we use $\Z_n$ to denote the finite cyclic group $\Z/n\Z$.
Given a surface~$F$, we let~$g(F)$ denote its genus.

\subsection*{Acknowledgements}
The second author thanks Federico Cantero Mor\'{a}n and Jason Joseph for discussions on Theorem~\ref{thm:A}.
Both authors thank the referee for a careful reading and many valuable comments.
During the preparation of this paper, the first author was partially supported by NSF grant DMS-1902880.

\section{Stabilization distances}\label{subsection:stab-distances}

Fix a compact, oriented, smooth 4-manifold $W$.
The following definition is motivated by that of Juh{\'a}sz and Zemke~\cite{Juhasz-Zemke-stab-dist}.

\begin{defn}\label{defn:1handlestab}
Let $\Sigma$ be an oriented surface with boundary, smoothly and properly embedded in $W$.
Let $B$ be an embedding of $D^4$ into $W$ such that $\partial B$ intersects $\Sigma$ transversely in a 2-component unlink $L$ and $B$ intersects $\Sigma$  in two discs $\Delta_0$ and $\Delta_1$, which can be simultaneously isotoped within $B$ to lie in $\partial B$. Suppose that a 3-dimensional 1-handle $D^2 \times I$ is embedded into the interior of $W$ such that $D^2 \times \{i\}= \Delta_i$ for $i=0,1$.
Then $\Sigma':= (\Sigma \cap (W \ssm B))  \cup_L (S^1 \times I)$ is a \emph{1-handle stabilization} of $\Sigma$.
If  $S^1 \times I$ can be isotoped into $\partial B$ relative to $L$, we call the stabilization \emph{trivial}.
\end{defn}
\begin{figure}[h]
\includegraphics[height=1.8cm]{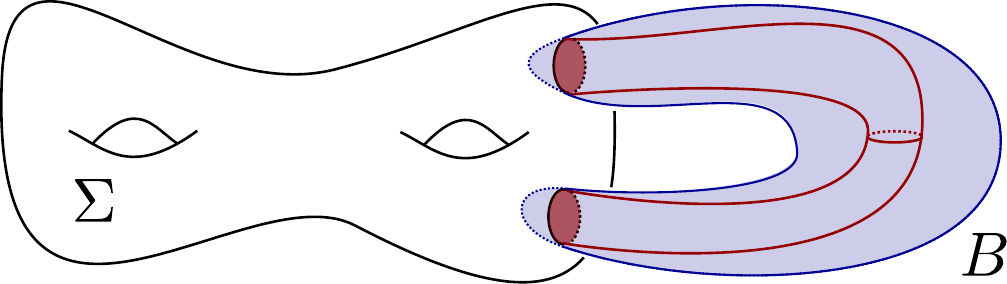}
\caption{A surface $\Sigma$ with ball $B$ as in Definition~\ref{defn:1handlestab}, pre-stabilization.}\label{fig:stabelabpic}
\end{figure}

%Any two surfaces $F$ and $F'$ in $D^4$ with $\partial F= K = \partial F'$ become isotopic after sufficiently many trivial 1-handle additions. Restricting to surfaces of the same genus, it is natural to define the \emph{trivial handle stabilization distance} between $F$ and $F'$ to be the minimal $k \in \mathbb{N}$ such that $F$ and $F'$ are isotopic after each is trivially stabilized $k$ times. We denote this by $d_{t}(F, F')$.

A trivial 1-handle stabilization does not change the fundamental group of the complement of the surface, so frequently there will be no sequence of trivial stabilizations relating two given surfaces.  On the other hand, any two homologous surfaces become isotopic after adding finitely many 1-handles~\cite{Bay-Sun-15}.

\begin{defn}
Define the \emph{1-handle stabilization distance} in $\N \cup \{0,\infty\}$ between smoothly and properly embedded surfaces $(F,\partial F) \subset (W,\partial W)$ and $(F',\partial F') \subset (W,\partial W)$ with $\partial F=\partial F'$, homologous in $H_2(W,\partial W;\Z)$, to be the minimal $k  \in \mathbb{N}$ such that $F$ and $F'$ become ambiently isotopic rel.\ boundary after each has been stabilized  at most $k$ times. We denote this by $d_{1}(F, F')$.  If $F$ and $F'$ are not homologous or have different boundaries then we say that $d_{1}(F, F') = \infty$.
\end{defn}

In particular for any two 2-knots $K$ and $J$, $d_{1}(K,J) < \infty$.
For distances between slice discs, we obtain stronger results by defining a coarser notion that permits connected sum with locally knotted 2-spheres.
%\begin{defn}\label{defn:adding-local-2-knot}
%Let $\Sigma$ be an oriented surface with boundary that is smoothly and properly embedded in $W$.
%Let $B$ be an embedding of $D^4$ into the interior of $W$ such that $\partial B$ intersects $\Sigma$ transversely in an unknot $U$ and $B$ intersects $\Sigma$  in a disc $\Delta$ that can be isotoped within $B$ to lie in $\partial B$.
%
%Let $S \subset S^4$ be a 2-knot, and suppose there is an embedding $B'$ of $D^4$ into $S^4$ such that $\partial B'$ intersects $S$ transversely in an unknot $U$ and $B'$ intersects $S$  in a disc $\Delta'$ that can be isotoped within $B'$ to lie in $\partial B'$.
%Then
%\[\Sigma':= (\Sigma \cap (W \ssm int(B)))  \cup_U (S \cap (S^4 \ssm \mathring{B}')) \subset (W \ssm \mathring{B}) \cup_{\partial B = \partial B'} (S^4 \ssm \mathring{B}') \cong   W \# S^4 \cong W\] is the result of \emph{adding the locally knotted 2-sphere $S$} to $\Sigma$ and denoted $\Sigma \# S$.
%\end{defn}
%\footnote{\allison{ A less detailed option might be to say that
By \emph{adding a locally knotted 2-sphere} to a properly embedded surface $(\Sigma,\partial \Sigma) \subset (W,\partial)$ we mean taking a 2-knot $S$ in $S^4$ and forming the connected sum of pairs
\[(W, \Sigma) \#(S^4, S) = (W, \Sigma \#S).\]
%That said, I don't think there's anything wrong with being very careful about definitions.}}

\begin{defn}
Let $(F,\partial F) \subset (W,\partial W)$ and $(F',\partial F') \subset (W,\partial W)$ be smoothly and properly embedded surfaces.
If $\partial F=\partial F'$ and  $[F]=[F']\in H_2(W, \partial W; \Z)$, we define the \emph{generalized stabilization distance} $d_2(F, F')$ in $\N \cup \{0,\infty\}$
  to be the minimal $k \in \mathbb{N}$ such that $F$ and $F'$ become ambiently isotopic rel.\ boundary after each has been stabilized at most $k$
  times and had arbitrarily many locally knotted 2-spheres added. If $F$ and $F'$ are not homologous or have different boundaries then we say that $d_2(F, F') = \infty$.
\end{defn}
%{AM: Note (to self) that we are only allowing 2-sphere \emph{addition}, not addition and removal. MP: Maybe this should be a note to all as well? AM: Maybe? I think it's the same definition, since if $F$ is isotopic to $F'$ with a knotted 2-sphere removed, then $F$ with a knotted 2-sphere added is isotopic to $F'$ (though perhaps the converse is not true?). I was mostly just reminding myself that in the proof of Theorem~\ref{thm:B} we don't have to think about removal.}
%{AM: We could insert a comment here about the fact that we require our 1-handle additions to be smooth, but allow both our isotopies of surfaces and 2-knots to just be locally flat (which I think is correct). We use smoothness, so far as I can tell, only in Construction~\ref{cons:standardcobordism} to get a handle structure on $X_T$. So our results hold somewhat more generally than as is stated here.}

Note that for any two slice discs $D_1,D_2$ in $D^4$ for a fixed knot in $S^3$, we have that $d_2(D_1,D_2) < \infty$.
It is immediate from the definitions that
\[ d_2(F,F') \leq d_{1}(F, F').\]
We also remark that $d_{JZ}(F,F') \leq d_2(F,F')$, where $d_{JZ}$ denotes the Juh{\'a}sz-Zemke stabilization distance~\cite{Juhasz-Zemke-stab-dist} between surfaces.

%Our approach to studying the distances between slice discs will be, given a surface $F$ with boundary $K$, to analyse $\ker(H_1^*(X_K) \to H_1^*(X_F)$ with respect to some twisted coefficient system, and determine how this kernel changes under 1-handle stabilization of $F$.

\section{Cobordisms corresponding to handle additions}\label{section:cobordisms}

Now we construct cobordisms corresponding to handle additions.  The following construction will be used in our proofs of all three main theorems.

\begin{cons}\label{cons:standardcobordism}[A cobordism between surface exteriors.]
Let $W$ be a compact, oriented, smooth 4-manifold.
Suppose that $F_1$ is a smoothly and properly embedded surface in $W$ with $\partial F_1=K \subset \partial W$ and that $F_2$ has been obtained from $F_1$ by a 1-handle addition such that $g(F_2)= g(F_1)+1$.
We define an ambient cobordism $T \subset W \times I$ as follows:
\begin{align*}
T:= (F_1\times [0,1/2]) \cup ((D^1 \times D^2) \times \{1/2\}) \cup (F_2 \times [1/2, 1]),
\end{align*}
where $D^1 \times D^2 \hookrightarrow W$ is an embedding with $\partial D^1 \times D^2 \subset F_1$ and $D^1 \times \partial D^2 \subset F_2$.  (That is,~$D^1 \times D^2$ is the 3-dimensional 1-handle~$h$ in the definition of 1-handle stabilization.)
Observe that
\[ \partial T= (F_1 \times \{0\}) \cup_{K \times \{0\}} (K \times [0,1]) \cup_{K \times \{1\}} F_2 \times \{1\}
\]
and so $X_T\colon=(W \times I)\ssm \nu(T)$ is a cobordism rel.\ $X_K$ from $X_{F_1}$ to $X_{F_2}$.

Since $T$ is obtained from $F_1 \times [0,1/2]$ by attaching a single 3-dimensional 1-handle to $F_1 \times \{1/2\}$ (and then flowing upwards), it follows from the rising water principle~\cite[Section~6.2]{Gompf-stipsicz-book} that $X_T$ has a handle decomposition relative to  $X_{F_1}$ obtained by attaching a single 5-dimensional 2-handle to $X_{F_1} \times I$. Notice that the attaching sphere of this 2-handle determines an element of $\pi_1(X_{F_1})$ of the form $\gamma=\mu_1 \beta \mu_2^{-1} \beta^{-1}$, where $\mu_1$ and $\mu_2$ are meridians to $F_1$ near the attaching spheres of $h$ and $\beta$ is a parallel push-off of the core of $h$. In particular, $\gamma$ is null-homologous in $H_1(X_{F_1})$.
Taking the dual decomposition, we see that $X_T$ also has a handle decomposition relative to $X_{F_2}$ obtained by attaching a single 5-dimensional 3-handle.
By excision, we therefore have that
\[ H_k(X_T, X_{F_1})=\begin{cases} \Z & k=2\\ 0 & \text{else} \end{cases} \text{ and  } H_k(X_T, X_{F_2})=\begin{cases} \Z & k=3 \\ 0 & \text{else.} \end{cases}  \]
In particular, the inclusion maps $X_{F_i} \to X_T$ induce isomorphisms on first homology.
It will be useful for us later on to know that the inclusion induced map $\pi_1(X_{F_1}) \to \pi_1(X_{T})$ is surjective, as follows immediately from applying the Seifert-van Kampen theorem to $X_T= (X_{F_1} \times I) \cup (2\text{-handle})$.

We now comment on basepoints for the fundamental group in this context.  Let $x_0 \in X_K \subseteq X_T \times \{0\}$, let $\alpha=\{x_0\} \times I \subseteq X_T \times I$, and let $x_1= \{x_0\} \times 1$.  We will always let $\pi_1(X_K)= \pi_1(X_K, x_0)$, $\pi_1(X_{F_1})= \pi_1(X_{F_1}, x_0)$, $\pi_1(X_T)= \pi_1(X_T, x_0)$, and $\pi_1(X_{F_2})= \pi_1(X_{F_2}, x_1)$.
There are natural inclusion induced maps $\iota \colon \pi_1(X_K, x_0) \to \pi_1(X_T, x_0)$ and $\iota_1\colon  \pi_1(X_{F_1}, x_0) \to \pi_1(X_T, x_0)$. Moreover, we use the arc $\alpha$ to define
\[\iota_2 \colon \pi_1(X_{F_2}, x_1) \to \pi_1(X_T, x_1) \to \pi_1(X_T ,x_0).\]
Later on, we will often omit basepoints from our notation, always using the above arcs and corresponding inclusion maps. This completes Construction~\ref{cons:standardcobordism}.
\end{cons}

\begin{prop}
  Fix a compact, oriented, smooth 4-manifold $W$, a $($possibly empty$)$ link~$L$ in $\partial W$, a nonnegative number $g$, and  a homology class $A \in H_2(W,\partial W;\Z)$ with $\partial A=[L]$.  The distance function~$d_{1}$ defines a metric on the set of ambient isotopy classes rel.\ boundary of embedded oriented surfaces of genus $g$ in $W$ with boundary $L$ that represent the class $A \in H_2(W,\partial W;\Z)$.
  \end{prop}
%\footnote{AM: Do we want to comment on whether $d_2$ is a metric? MP: Good idea. What can we say? It's not zero if and only if isotopic. So it's at best a pseudo metric. Can we say it's a metric on a set of isotopy classes of suitable surfaces modulo local knots? AM: The vanishing condition seems right. Less sure about proving the triangle inequality, unless we just change representative of the equivalence class so that no adding of 2-spheres is necessary?}

\begin{proof}
  We use that the distance is finite within the sets considered~\cite{Bay-Sun-15}.   If $d_{1}(\Sigma,\Sigma')=0$, then~$\Sigma$ and~$\Sigma'$ are ambiently isotopic.  The distance function is flagrantly symmetric.

  To see the triangle inequality,
 suppose $F$ and $F'$ are homologous rel.\ boundary surfaces which stabilize via $k$ 1-handle additions to a surface $S$ and $F'$ and $F''$ are homologous rel.\ boundary surfaces which  stabilize via $h$ 1-handle additions to $S'$.
Now consider the sequence of stabilizations and destabilizations  from $F$ to $S$ to $F'$ to $S'$ to $F''$ as a 3-dimensional cobordism $T$
 embedded in $W \times I$.  We may perturb the embedding of $T$ so that  $F \colon W \times I \to I$ restricts to a Morse function on $T$, where stabilizations correspond to index one critical points, and destabilizations correspond to index two critical points. First we argue that we can rearrange this sequence of stabilizations and destabilizations so that all the stabilizations come first, followed by destabilizations. Our desired result will then follow immediately from letting $S''$ be the preimage of a regular value taken after all index one critical points and before all index two critical points, and observing that both $F$ and $F''$ stabilize via $(k+h)$ 1-handle additions to $S''$.

In codimension at least two, critical points of an embedded cobordism can be arranged, by ambient isotopy, to appear in order of increasing index~\cite{Pe}, \cite[Theorem~4.1]{BP}, by the following standard argument, which we include for completeness.
  Choose a gradient-like embedded vector field subordinate to $F$~\cite[Definition~3.1]{BP}.  Rearrangement of critical points is possible in general if the ascending manifold of the lower critical point is disjoint from the descending manifold of the higher critical point. Suppose that an index one critical point of $T$ has critical value $t_1$  higher than critical value $t_2$ of an index two critical point, and suppose that there are no critical values between $t_2$ and $t_1$. The descending manifold of the index 1 critical point of a 3-dimensional cobordism intersects a generic level set $W \times \{t\}$, with $t_2 < t < t_1$ in a 1-dimensional disc. The descending manifold of the index 2 critical point intersects  $W \times \{t\}$ also in a 1-dimensional disc. By general position, we can perturb the gradient-like vector field to make the ascending and descending manifolds disjoint, and we may do so simultaneously for all such~$t$.  It follows that the critical points can be rearranged by an ambient isotopy, as desired.
\end{proof}

We remark that we do not claim $d_2$ gives rise to a metric.
The next proposition tells us that $2$-spheres can be reordered so they come before 1-handle additions.

\begin{prop}\label{prop:2spheresfirst}
Suppose that an embedded surface $\Sigma_2$ is obtained from a connected surface~$\Sigma_1$ by some number $m$ of 1-handle additions, followed by connect summing with a local 2-knot.  Then there is an embedded surface~$\Sigma'$ that is obtained from~$\Sigma_1$ by adding a local 2-knot, and such that~$\Sigma_2$ is obtained from~$\Sigma'$ by~$m$ 1-handle additions.
\end{prop}
\begin{proof}
Let $\Sigma_1'$ denote $\Sigma_1$ with the 1-handles attached, so $\Sigma_2$ is obtained from $\Sigma_1'$ by connected sum with a local 2-knot $S$.
The isotopy class of $\Sigma_1' \# S$ is unchanged by where on $\Sigma_1'$ we take the connected sum, so we can assume that our connected sum  takes place far away from the attached 1-handles. But then it is clear that we can attach $S$ first and our 1-handles second.
\end{proof}

\section{Generating ranks and twisted homology}\label{section:gen-rank}
%\footnote{\allison{Perhaps a stupid title. Feel free to change.}}

\subsection{Generating rank of modules over a commutative PID}
We recall some facts about generating ranks of finitely generated modules over
commutative PIDs.

Let $A$ be a finitely generated module over a commutative PID~$S$.
We say that $A$ has \emph{generating rank  $k$ over $S$} if $A$ is
    generated as an $S$-module by $k$ elements but not by $k-1$ elements and  write
     $\gr_S A=k$. When $S$ is clear from context, we often abbreviate $\gr_S A$ by $\gr A$.

\begin{lem}\leavevmode\label{lemma:gen-rank-facts}
Let $A$, $B$, and $C$ be finitely generated modules over a commutative PID~$S$.
\begin{enumerate}
  \item\label{item:gen-rank-surjection} If $A$ surjects onto $B$ then $\gr_S B \leq \gr_S A$.
  \item\label{item:gen-rank-subgroup} If $B \leq A$ then $\gr_S B \leq \gr_S A$.
\item\label{item:gen-rank-ses} Let $0 \to A \xrightarrow{f} B \xrightarrow{g} C \to 0$ be a short exact sequence of $S$-modules. Then \\ $\gr_s C \geq \gr _S(B) - \gr_S(A)$.
\end{enumerate}
\end{lem}

\begin{proof}
The first part follows immediately from the definition of generating rank.
The second part is easy to check using the classification of finitely generated modules over a commutative PID.
The third property follows from taking minimal $S$-generating sets $\{a_1, \dots, a_n\}$ and $\{c_1, \dots, c_m\}$    for $A$ and $C$ respectively,
picking $b_i \in g^{-1}(c_i)$ for each $1 \leq i \leq m$, and observing that $\{f(a_1), \dots, f(a_n), b_1, \dots, b_m\}$ is an $S$-generating set for $B$.
\end{proof}

\begin{remark}
  Only \eqref{item:gen-rank-subgroup} uses that $S$ is a PID.
\end{remark}

We will also make arguments involving the \emph{order} of  a finitely generated module~$A$ over a  commutative PID~$S$.
The classification of finitely generated modules over a PID states that
there exist $j, k \in \mathbb{N}$ and  elements $s_1, \dots, s_k \in S$ such
that there is a (non-canonical) isomorphism \[A \cong S^j \oplus TA \cong
    S^j \oplus \bigoplus_{i=1}^k S/ \langle s_i \rangle.\] When $j>0$ we say
    that the \emph{order} of $A$ is $|A|=0$ and when $j=0$ we say that the order
    of $A$ is $|A|= \prod_{i=1}^k s_i$. This is well-defined up to
    multiplication by units in $S$. The key property of order we use is that if
    $f\colon A \to B$ is a map of $S$-modules with $\ker(f)$ torsion, then
    $|\Imm(f)|= |A|/|\ker(f)|$.

\subsection{Twisted homology}

Let $X$ be a CW complex with universal cover $\wt{X}$.  The cellular chain complex $C_*(\wt{X})$ is a chain complex of right  $\Z[\pi_1(X)]$-modules. If $X$ is a finite complex then $C_*(\wt{X})$ is finitely generated as a $\Z[\pi_1(X)]$-module.  Let $R$ be a commutative ring with involution and with unit.  Let $\alpha \colon \pi_1(X) \to U_m(R)$ be a unitary representation i.e.\ $\alpha(g^{-1}) = \ol{\alpha(g)}^T$. This extends to a homomorphism of rings with involution $\Z[\pi_1(X)] \to \GL_m(R)$, and makes $R^m$ into a $(\Z[\pi_1(M)],R)$-bimodule.

\begin{defn}
  The $k$th \emph{twisted homology} of $X$ with respect to $\alpha$ is \[H_k^{\alpha}(X;R) := H_k(C_*(\wt{X}) \otimes_{\Z[\pi_1(X)]} R^m).\]
\end{defn}
%\footnote{AM: Should there really be a superscript $\alpha$ on the right above? MP: I guess not, deleted.}

When the ring $R$ is clearly understood, and we are short of space, we shall sometimes omit $R$ from the notation and write $H_k^{\alpha}(X)$ for $H_k^{\alpha}(X;R)$.

If $X$ is a finite complex and $R$ is Noetherian then $H_k^{\alpha}(X;R)$ is finitely generated as an $R$-module.
If $Y \subset X$ is a subcomplex and we choose a path $\gamma \colon I \to X$ from the basepoint then $\alpha$ determines a representation $\pi_1(Y) \to U_m(R)$ and we write $H_k^{\alpha}(Y;R)$ for the resulting twisted homology. The inclusion induced map $H_k^{\alpha}(Y;R) \to H_k^{\alpha}(X;R)$ depends on the choice of $\gamma$, but nonetheless we omit $\gamma$ from the notation.

\begin{remark}
Given $X$ and $\alpha \colon \pi_1(X) \to U_m(R)$ as above, let $X^{\alpha} \to X$ be the cover corresponding to $\ker(\alpha)$. Then $\Z[\pi_1(X)]$ acts on $C_*(X^\alpha)$ and it follows immediately from our definitions that
\[H_k^{\alpha}(X;R) \cong H_k(C_*(X^{\alpha}) \otimes_{\Z[\pi_1(X)]} R^m).\]
It is sometimes more convenient to compute with this smaller covering space.
\end{remark}
%\footnote{AM: This is what I meant by  induced cover. MP: do we need it now? AM: I think we do use it e.g. when thinking of the Alexander module as twisted homology yet still computing it via cover?}

\subsection{Rational Alexander modules}

For any knot or slice disc $L$, let $\A(L)$  denote the Alexander module of $L$ with integral coefficients and let $\A_{\Q}(L)$ denote the Alexander module of $L$ with rational coefficients. That is, let $X_L$ be the exterior of $L$ and as usual let $\varepsilon \colon \pi_1(X_L) \to \Z$ denote the abelianization map.
Then $\A(L):= H_1(X_L, \Z[t^{\pm1}])$ and $\A_{\Q}(L):= H_1(X_L; \Q[t^{\pm}])$, where for $R=\Z,\Q$ the ring $R[t^{\pm1}]$ has a $\Z[\pi_1(X_L)]$-structure determined by $\varepsilon$.  We remark that $\Q$ is flat as a $\Z$-module, and so $\A_{\Q}(L) \cong \A(L) \otimes_{\Z} \Q$.
%\footnote{AM: Added definition of $\A(L)$!}

\section{Pairs of 2-knots with arbitrary 1-handle distance}\label{section:2-knots}

In this section, we prove  that for every nonnegative integer $m$, there exists a pair of 2-knots $K$ and $J$ in the $4$-sphere with 1-handle stabilization distance $m$, which is  an immediate consequence of the  following proposition.

\begin{prop}\label{prop:largedistancetotrivial}
For each $m \in \mathbb{N}$, there exists a knotted 2-sphere $K$ in $S^4$ such that the minimal number of 1-handle stabilizations needed to make $K$ an unknotted surface is
exactly~$m$.
\end{prop}

\begin{proof}[Proof of Theorem~\ref{thm:A}]
Let $m \in \mathbb{N}$, let $K$ be as in Proposition~\ref{prop:largedistancetotrivial}, and let $J$ be an unknotted 2-sphere.
Since every stabilization of an unknotted 2-sphere is an unknotted surface,
we obtain immediately that $d_1(K,J)=m$.
\end{proof}

The next proposition is the key algebraic input into the proof of Proposition~\ref{prop:largedistancetotrivial}.

\begin{prop}\label{prop:change-in-alex-module-2-knot}
Let $F_1 \subset S^4$ be a smoothly embedded oriented surface and suppose that $F_2$ is obtained from $F_1$ by a 1-handle stabilization.
Then there is a polynomial $p \in \Q[t^{\pm 1}]$ and a short exact sequence
  \[0 \to \Q[t^{\pm 1}]/\langle p \rangle \to H_1(S^4 \ssm \nu F_1;\Q[t^{\pm 1}]) \to H_1(S^4\ssm \nu F_2;\Q[t^{\pm 1}]) \to 0.\]
\end{prop}

\begin{proof}
  We consider the relative cobordism $X_T$ between $X_{F_1}$ and $X_{F_2}$ from Construction~\ref{cons:standardcobordism}, with $W=S^4$. We will consider  the infinite cyclic cover $\widetilde{X}_T$.  Recall that $X_T$ is obtained from $X_{F_1} \times I$ by attaching a single 5-dimensional 2-handle along $\gamma \times \{1\}$ for  $\gamma= \mu_1 \beta \mu_2^{-1} \beta^{-1}$, where $\mu_1$ and $\mu_2$ are meridians of $F_1$ in $S^4$ near the attaching spheres of the 1-handle and $\beta$ is a parallel push-off of the core of this 1-handle.  Since $H_1(F_1;\Z) \cong \Z$, and the attaching sphere of the 2-handle is null homologous, the abelianization homomorphism $\pi_1(X_{F_1}) \to \Z$ extends to a homomorphism $\pi_1(X_T) \to \Z$.
  From now on in this proof we consider homology with $\Q[t^{\pm 1}]$-coefficients induced by this homomorphism.   We also note that the handle decomposition lifts to a relative handle decomposition of $\widetilde{X}_T$ with one orbit of 2-handles under the deck transformation action of~$\Z$.

Using this relative handle decomposition we  obtain that $H_k(X_T, X_{F_1}; \Q[t^{\pm 1}])=0$ for $k \neq 2$ and $H_2(X_T, X_{F_1}; \Q[t^{\pm 1}]) \cong \Q[t^{\pm 1}]$. Since dually $X_T$ is obtained from $X_{F_2} \times I$ by attaching a single 5-dimensional 3-handle, we have that $H_k(X_T,  X_{F_2}\Q[t^{\pm 1}])=0$ for $k \neq 3$.
Now consider the long exact sequence of the pair $(X_T,X_{F_1})$ with $\Q[t^{\pm 1}]$-coefficients.
     \[\cdots \to H_2(X_T) \to H_2(X_T,X_{F_1}) \to H_1(X_{F_1}) \to H_1(X_T) \to H_1(X_T,X_{F_1}).\]
Since $H_1(X_T,X_{F_1}) =0$ and $H_2(X_T,X_{F_1}) \cong \Q[t^{\pm 1}]$, and since $\Q[t^{\pm 1}]$ is a PID, this yields a short exact sequence
     \[0 \to \Q[t^{\pm 1}]/ \langle p \rangle \to H_1(X_{F_1}) \to H_1(X_T) \to 0\]
for some $p \in \Q[t^{\pm 1}]$.
Now the long exact sequence of the pair $(X_T,X_{F_2})$ yields
     \[0=H_2(X_T,X_{F_2}) \to H_1(X_{F_2}) \to H_1(X_T) \to H_1(X_T,X_{F_2})=0,\]
from which it follows that the inclusion induced map $H_1(X_{F_2}) \to H_1(X_T)$ is an isomorphism, and so we obtain the desired short exact sequence
     \[0 \to \Q[t^{\pm 1}]/ \langle p \rangle \to H_1(X_{F_1}) \to H_1(X_{F_2}) \to 0. \qedhere\]
\end{proof}

For the reader's convenience, we now describe two common constructions of slice discs.
\begin{cons}\label{cons:slicediscs}
Given a subset $Y \subseteq S^3$ and $J \subseteq I$ that is either an interval $[a,b]$ or a point $\{a\}$, write $Y_J$ for $Y \times J \subseteq S^3 \times I$.
We think of $D^4$ as $D^4 \cong S^3_{[0,1]}/ S^3_ {1}$.

\textit{The banding construction.}
Let $K$  be a knot with disjointly embedded bands $\beta_1, \dots, \beta_n$ in $S^3$ such that the result of banding $K$ via $\{\beta_i\}_{i=1}^n$ is the $(n+1)$-component unlink $U_{n+1}$, which could be capped off via $(n+1)$ discs in $S^3$.
Then, up to smoothing corners,
\[D:= K_{[0,1/3]} \cup   \left(\cup_{i=1}^n \beta_i\right)_{1/3} \cup  (U_{n+1})_{[1/3, 2/3]} \cup   \left(\cup_{i=1}^{n+1} D^2\right)_{2/3} \]
is a ribbon disc for $K$.

\textit{The surgery construction.}
Let $K$ be a knot with a genus $g$ Seifert surface $F$ and a collection of $g$ disjoint curves $\alpha_1, \dots, \alpha_g \subset F$ which are 0-framed by $F$ and which generate a $\Z^g$ summand of $H_1(F)$. Suppose also that the link $\cup_{i=1}^g \alpha_i \subset S^3$ is an unlink. Then, up to smoothing corners,
\[ D= K_{[0,1/3]} \cup (F \smallsetminus \nu(\cup_{i=1}^g \alpha_i))_{1/3}  \cup \cup_{i=1}^g (\alpha_i^+ \sqcup \alpha_i^-)_{ [1/3, 2/3]} \cup \cup_{i=1}^n (D^2 \sqcup D^2)_{2/3} \]
is a ribbon disc for $K$. We note that this construction is easily adapted to build a slice disc for $K$ under the weaker assumption that $\cup_{i=1}^g \alpha_i$ is merely strongly slice.
\end{cons}

\begin{exl}[The knot $9_{46}$ and its two standard slice discs.]\label{exl:946}
Let $R:=9_{46}$, and let $D_j$ for $j=1,2$ be the slice discs indicated by the left and right bands, respectively, of the left part of Figure~\ref{fig:946}.
\begin{figure}[ht]
  \includegraphics[height=5cm]{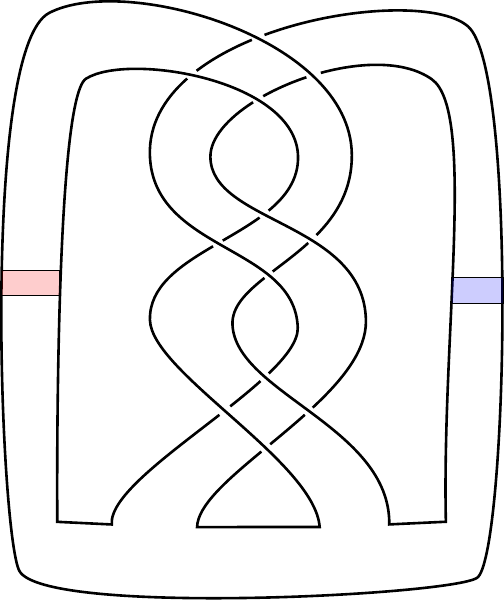}\hspace{1cm}
  \includegraphics[height=5cm]{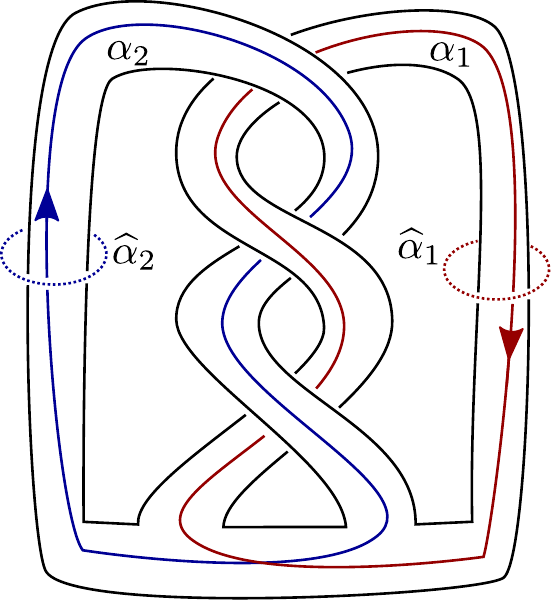}
  \caption{The knot $R=9_{46}$ has slice discs  $D_1$ (left band) and $D_2$ (right band).}
  \label{fig:946}
\end{figure}
Observe that $R$ has a genus 1 Seifert surface $F$ (illustrated on the right of Figure~\ref{fig:946}), and for $j=1,2$ let $D_j'$ be the slice disc  obtained by surgery of $F$  along $\alpha_j$.
Referring back to Construction~\ref{cons:slicediscs} for our explicit description of $D_j$ and $D_j'$, we can recognize these as isotopic discs in $D^4$, since
 \[R_{[1/6,1/3]} \cup   (\beta_j )_{1/3}  \, \cup\,  (U_2^j)_{[1/3, 2/3]} \subset D_j \text{ and } R_{[1/6, 1/3]} \cup (F \smallsetminus \nu(\alpha_j))_{1/3}  \cup (\alpha_j^+ \sqcup \alpha_j^-)_{[1/3, 2/3]} \subset D_j'\]
 are isotopic rel.\ boundary
as subsets of $S^3 \times [1/6, 2/3]$.

The oriented curves $\alpha_1, \alpha_2$  represent a basis for $H_1(F)$ with respect to which the Seifert pairing is given by
\[A=\left[\begin{array}{cc}0& 2 \\ 1 & 0 \end{array} \right].\] The Alexander module is therefore presented by \[tA-A^ T=\left[\begin{array}{cc}0& 2t-1 \\ t-2 & 0 \end{array} \right],\] and hence is isomorphic to $\Z[t^{\pm 1}]/ \langle t-2 \rangle \oplus \Z[t^{\pm1}]/ \langle 2t-1 \rangle $, where $\widehat{\alpha}_1$ and $\widehat{\alpha}_2$ represent the generators of each summand.

Moreover, the inclusion induced maps $\iota_j \colon \A_{\Q}(R) \to \A_{\Q}(D_j)$ are given by projection onto summands:
\begin{align*}
\A_{\Q}(R) \cong \Q[t^{\pm1}]/ \langle 2t-1 \rangle \oplus  \Q[t^{\pm1}]/ \langle t-2 \rangle
&\xrightarrow{\iota_1} \Q[t^{\pm1}]/ \langle 2t-1 \rangle \cong \A_{\Q}(D_1) \\
(x,y) &\mapsto x\\
\A_{\Q}(R) \cong \Q[t^{\pm1}]/ \langle 2t-1 \rangle \oplus  \Q[t^{\pm1}]/ \langle t-2 \rangle
&\xrightarrow{\iota_2} \Q[t^{\pm1}]/ \langle t-2\rangle \cong \A_{\Q}(D_2)  \\
(x,y) &\mapsto y.
\end{align*}
Note that $\ker(\iota_1) \cap \ker(\iota_2)= \{0\} \subseteq \A_{\Q}(R)$.

A detailed computation with these slice discs can be found in \cite[Section~5.1]{Conway-Powell}. To see that the induced maps are as claimed, we argue by  the rising water principle~\cite[Section~6.2]{Gompf-stipsicz-book}. There is a handle decomposition of $X_{D_i}$ relative to $X_R$ consisting of one 2-handle attached along $\widehat{\alpha}_i$ (corresponding to the band), followed by two 3-handles corresponding to the maxima, and a 4-handle. Only the 2-handle affects first homology, by killing the class represented by $\widehat{\alpha}_i$.
\end{exl}

\begin{proof}[Proof of Proposition~\ref{prop:largedistancetotrivial}]
Let $D:=D_2 \subset D^4$ be the ``right band'' slice disc for the $9_{46}$ knot shown via a blue band on the left of Figure~\ref{fig:946}. Let $K_0$ be the 2-knot obtained from doubling this disc, that is $K_0= D \cup_{9_{46}} D \subset D^4 \cup D^4 = S^4$.
Let $K := \#_{i=1}^{m} K_0$.

First we use Proposition~\ref{prop:change-in-alex-module-2-knot} to show that if $K$ stabilizes to an unknotted surface by $n$ 1-handle additions then $n \geq m$.
We know that
\begin{align*}
H_1(S^3 \ssm \nu(9_{46}); \Q[t^{\pm 1}]) &\cong \Q[t^{\pm 1}]/ \langle 2t-1 \rangle \oplus \Q[t^{\pm 1}]/ \langle t-2 \rangle
\end{align*}
where the inclusion induced map to $H_1(D^4 \ssm \nu(D);  \Q[t^{\pm 1}]) \cong \Q[t^{\pm 1}]/ \langle t-2 \rangle$ is given by projection onto the second factor.
By using the Mayer-Vietoris sequence corresponding to the decomposition
\[S^4 \ssm \nu K_0= \left( D^4 \ssm \nu(D)\right) \cup_{S^3 \ssm \nu(9_{46})}\left( D^4 \ssm \nu(D)\right),\]
we can compute that
\[H_1(S^4 \ssm \nu K_0;\Q[t^{\pm 1}]) = \Q[t^{\pm 1}]/(t-2).\]
Since Alexander modules are additive under connected sum of 2-knots we therefore have that
 \[ H_1(S^4 \ssm \nu K;\Q[t^{\pm 1}]) = \bigoplus_{i=1}^m \left( \Q[t^{\pm 1}]/(t-2)\right).\]

We therefore need to show that one requires at least $m$ stabilizations to trivialize the Alexander module of $K$.
Note that the generating rank of  $H_1(S^4 \ssm \nu K;\Q[t^{\pm 1}])$ is $m$.
We claim that the result of stabilizing an embedded surface whose Alexander module has generating rank $k$ is an embedded surface with generating rank at least $k-1$.
To see the claim, we use Proposition~\ref{prop:change-in-alex-module-2-knot} and the fact that if a $\Q[t^{\pm 1}]$-module $M$ has generating rank $k$ and a submodule $N$ has generating rank 1, then the quotient $M/N$ has generating rank at least $k-1$, by Lemma~\ref{lemma:gen-rank-facts}~(\ref{item:gen-rank-ses}). %Here we use that $\Q[t^{\pm 1}]$ is a PID.
By the claim and the fact that the generating rank of $H_1(S^4 \ssm \nu K;\Q[t^{\pm 1}])$ is $m$, it follows by induction that $d_{1}(K,J) \geq m$.
\\

It remains to show that we can make $K$ unknotted via $m$ 1-handle attachments.
Recall that the slice disc $D$ is constructed by a band move ``cutting'' one of the bands of the obvious Seifert surface $\Sigma$ for $9_{46}$ in Figure~\ref{fig:946}, and then capping off the resulting 2-component unlink with disjoint discs.  A single stabilization, tubing these two discs together, results in an embedded genus one surface. This surface could also be obtained  by capping off the 2-component unlink with an annulus instead of two discs, and hence is isotopic to the result of pushing the aforementioned Seifert surface into $D^4$.
We assert that $D \cup \Sigma \subset S^4$ is an unknotted genus one surface, and prove this by direct manipulation of  handle diagrams for the embedding of the surface in $D^4$, using the banded knot diagram moves  of Swenton~\cite{Swenton}.\footnote{The reader who is familiar with doubly slice knots may instead observe that $D \cup \Sigma$ is a stabilization of the unknotted 2-knot obtained by gluing the `left band' and `right band' discs together, and hence is itself unknotted.  We give the longer argument here to be self-contained.}

The data of an unlink and bands attached to it with the property that the result of performing the corresponding band moves is also an unlink provides instructions for embedding a surface in $S^4$: the unlink's components correspond to 0-handles, the bands to 1-handles, and the unlink obtained by banding can be capped off with 2-handles in an essentially unique way, in the sense that any two choices of discs in $S^3$ capping off the unlink yield isotopic surfaces in $S^4$. This uses the main result of \cite{livingston-surfaces},
that any two sets of embedded discs in $S^3$ are isotopic rel.\ boundary in~$D^4$. We remark that isotopy of banded knot diagrams in $S^3$ together with cancellation/ creation  of band-unknot pairs,  sliding of bands across each other,  and the `band-swim move' illustrated in Figure~\ref{fig:swimming} preserve the isotopy class of the presented surface (see Swenton~\cite{Swenton} for more details).
\begin{figure}[h]
\[
\begin{array}{ccc}
\includegraphics[align=c, height=2cm]{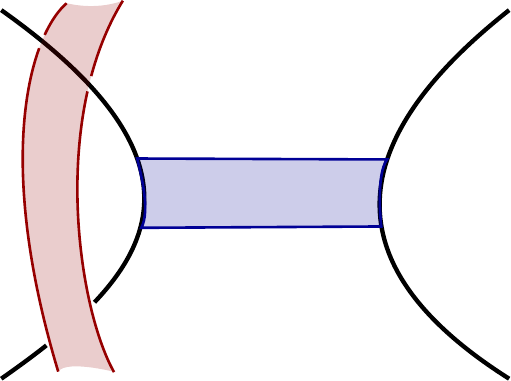}
& \leftrightarrow
& \includegraphics[align=c, height=2cm]{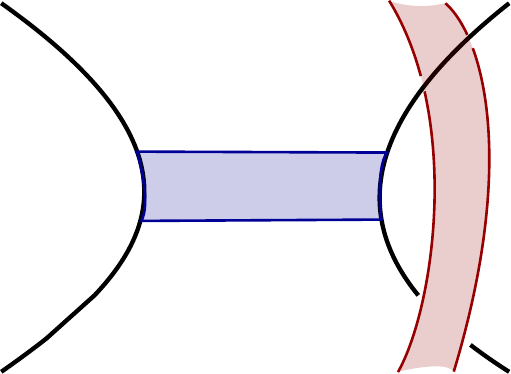}
\end{array}
\]
\caption{A  `band-swim' move preserves the isotopy class of a surface presented by a banded knot diagram.}
\label{fig:swimming}
\end{figure}

The banded diagram on the far left of Figure~\ref{fig:bandeddiagram} gives $D \cup \Sigma$.  The top two bands correspond to the Seifert surface, and the green band is the band of the disc~$D$.
\begin{figure}[h]
\includegraphics[align=t, height=3.3cm]{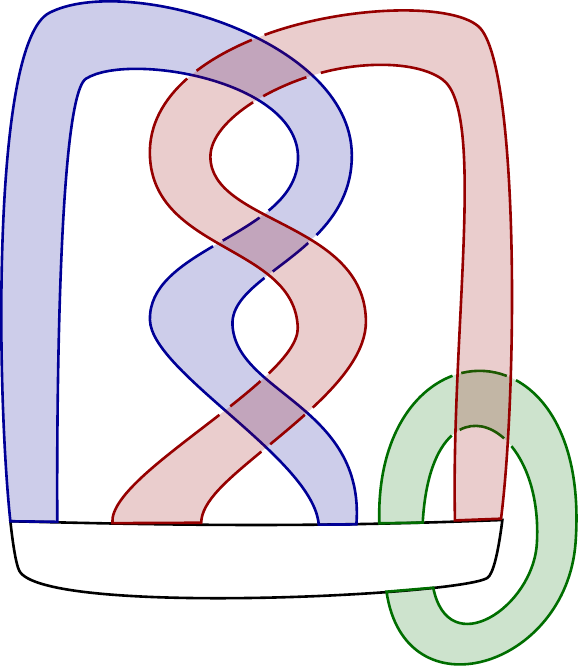}
\hspace{.5cm}
\includegraphics[align=t, height=3.3cm]{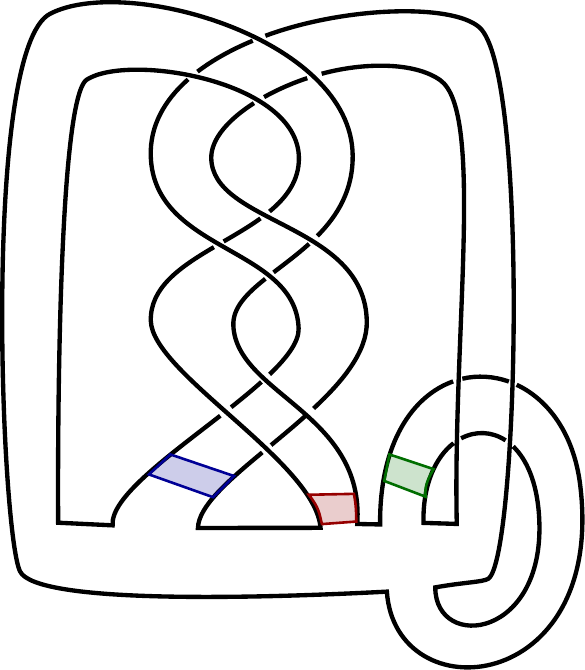}
\hspace{.5cm}
\includegraphics[align=t, height=3cm]{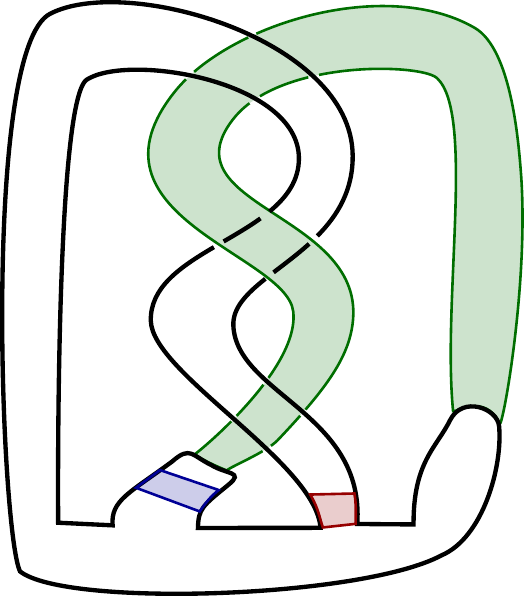}
\hspace{.5cm}
\includegraphics[align=t, height=3cm]{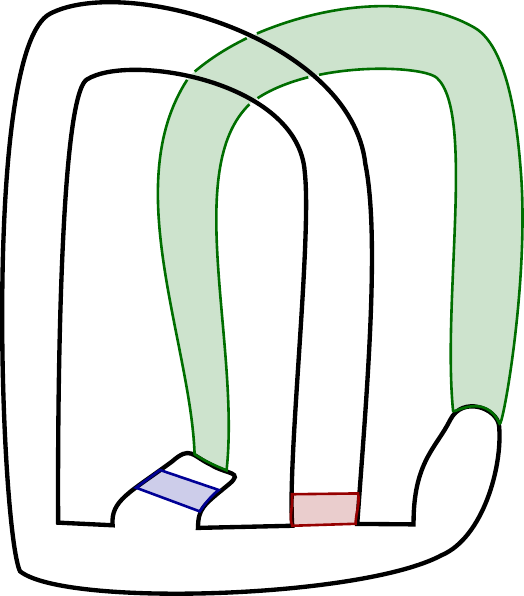}
\caption{Simplifying a banded knot diagram for $D \cup \Sigma$.}
\label{fig:bandeddiagram}
\end{figure}
%\footnote{What do you think about showing on each of the bands where it is attached to the knot? Actually it's quite clear in the colour on the computer but in my grey scale printing it's ambiguous. My desire for this information to be clearly shown is related to knowing exactly what a band swim is i.e. it's not just the same as a `crossing change' between bands.}
The  center left of Figure~\ref{fig:bandeddiagram} gives the `dual' band description corresponding to turning our handle diagram upside down.
The center right figure is obtained by an isotopy of the banded diagram in $S^3$, and we perform a `band-swim' move of the green band through the red band to obtain the diagram on the far right of Figure~\ref{fig:bandeddiagram}.

Now obtain the diagram on the left of Figure~\ref{fig:bd2} by an isotopy of the diagram in $S^3$, before sliding the green band across the red band to obtain the central diagram.
\begin{figure}[h]
\includegraphics[align=c, height=2cm]{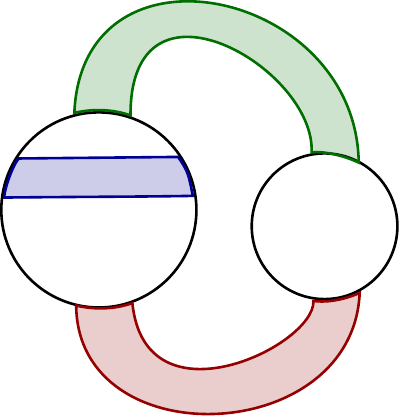}
\hspace{1cm}
\includegraphics[align=c, height=2cm]{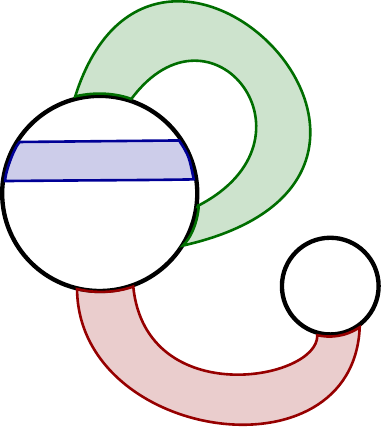}
\hspace{1cm}
\includegraphics[align=c, height=1.5cm]{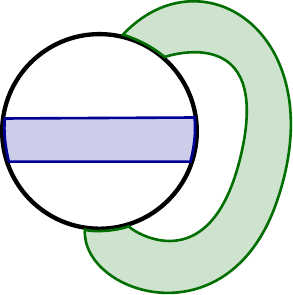}
\caption{Further simplifications of the banded knot diagram for $D \cup \Sigma$, resulting in the standard diagram for an unknotted torus (right).}
\label{fig:bd2}
\end{figure}
We can then cancel the right-hand unknot with the red band, corresponding to canceling a pair of 0- and 1-handles, in order to obtain the standard diagram for an unknotted torus seen on the right of Figure~\ref{fig:bd2}.
\end{proof}

\section{Pairs of slice discs with large generalized stabilization distance}\label{section:slice discs-1}

In this section we prove Theorem~\ref{thm:B}. We use the classical Alexander module to show that for every nonnegative integer $m$ there is a knot $K$ with slice discs $D$ and $D'$ such that $d_2(D, D')$ equals $m$.
To do this, we investigate the kernel of the induced map on fundamental groups from the knot exterior to the slice disc exteriors by using the homology of cyclic covering spaces.

First, we note that connected sum with a knotted 2-sphere has no effect on the kernel of the map on fundamental groups.

\begin{prop}\label{prop:2spheresum}
Suppose that $F_2$ has been obtained from $F_1$ by connected sum with a knotted 2-sphere $S$.
Then
\[ \ker( i_1\colon \pi_1(X_K) \to \pi_1(X_{F_1})) = \ker( i_2\colon \pi_1(X_K) \to \pi_1(X_{F_2})).\]
\end{prop}

\begin{proof}
Let $X_S:= S^4 \setminus \nu S$ be the exterior of $S$ in $S^4$.
Construct $X_{F_2}$ from $X_{F_1}$ and $X_{S}$ by identifying thickened meridians $S^1 \times D^2 \subset \partial X_{F_1}$ and  $S^1 \times D^2 \subset \partial X_{S}$ in the boundaries and smoothing corners.  By the Seifert-van Kampen theorem we have that
\[\pi_1(X_{F_2}) \cong \pi_1(X_{F_1}) *_{\Z} \pi_1(X_{S}).  \]
So $\pi_1(X_{F_1})$ is isomorphic to a subgroup of $ \pi_1(X_{F_2})$ in such a way that the inclusion-induced maps factor as
\[\pi_1(X_{F_1}) \hookrightarrow \pi_1(X_{F_1}) *_{\Z} \pi_1(X_{S}) \xrightarrow{\cong} \pi_1(X_{F_2}).\]
It follows that $\ker(i_1)=\ker(i_2)$.
\end{proof}

The following proposition is central to the rest of the paper, and so we state it in some generality. In particular, in later sections we will want to apply this result with twisted coefficients, so in the name of efficiency we state and prove the full version here.

\begin{prop}\label{prop:kerneloftwistedH1}
Let $F_1$ and $F_2$ be properly embedded  surfaces in $D^4$ with $\partial F_j= K$, where
$F_2$ has been obtained from $F_1$ by $g$  1-handle additions such that $g(F_2)= g(F_1)+g$.
Let $T\subseteq D^4 \times I$ be the 3-manifold  built as in Construction~\ref{cons:standardcobordism}.
Suppose that $\phi \colon \pi_1(X_K) \to \GL_m(R)$ extends over $\pi_1(X_T)$ to a map $\Phi \colon \pi_1(X_T) \to \GL_m(R)$.
%$($i.e.\ $\phi= \Phi \circ \iota)$, where $\iota$ is as in Construction~\ref{cons:standardcobordism}.).
For $j=1, 2$ define
\[P_j:= \ker\left( H_1^{\phi}(X_K;R) \to H_1^{\Phi} (X_{F_j}; R) \right).\]
Then $P_1\subseteq P_2$ and, assuming in addition that $R$ is a PID, $P_2$ is generated as an $R$-module by $P_1 \cup \{x_i\}_{i=1}^{gm}$ for some choice of  $x_i \in P_2$.
\end{prop}

\begin{proof} The case of general $g$ follows immediately from repeated application of the $g=1$ case, which we now prove.

%The map $\Phi$ induces a regular cover $\widetilde{X}_T \to X_T$, the cover corresponding to the normal subgroup $\ker(\Phi \colon \pi_1(X_T) \to \GL_m(R))$.  The homology $H_k^{\Phi}(X_T,X_{Y};R)$, where $Y = F_i,K$, is identified with \footnote{Can we instead work with the CW structure of (CW complex hom equiv to) $X_T$, and thence of the universal cover i.e.\ only 2-cells. Then this would elide this business with exactly which covering space we are looking at an how it related to the twisted homology. }

%We wish to lift the relative handle decomposition of Construction~\ref{cons:standardcobordism} to $\widetilde{X}_T$.

Recall that $X_T$ is obtained from $X_{F_1} \times I$ by attaching a single 5-dimensional 2-handle along $\gamma \times \{1\}$ for $\gamma$ a simple closed curve representing $[\gamma]= \mu_1 \beta \mu_2^{-1} \beta^{-1}$ in $\pi_1(X_{F_1})$, where $\mu_1$ and $\mu_2$ are meridians of $F_1$ in $D^4$ near the attaching spheres of the 1-handle, and $\beta$ is a parallel push-off of the core of this 1-handle.
%Our assumption that $\phi$ extends to a map $\Phi \colon \pi_1(X_T) \to \GL_m(R)$ implies that $\phi(\gamma) = \Id_{R^m}$.\footnote{MP: How necessary is this stuff about $\gamma$ now? AM: Why was it necessary to begin with? I forget.}

There is a CW pair $(X_T^{CW},X_{F_1}) \simeq (X_T,X_{F_1})$ where $X_T^{CW}$ is a CW complex obtained by attaching a single 2-cell to $X_{F_1}$ along $\gamma$. The universal cover $\wt{X}_T^{CW} \to X_T^{CW}$ induces a pull-back covering $\wt{X}_{F_1} \to X_{F_1}$, with relative cellular chain complex
\[C_*(\wt{X}_T^{CW},\wt{X}_{F_1}) \simeq C_*(\wt{X}_T,\wt{X}_{F_1})\]
with $C_2(\wt{X}_T^{CW},\wt{X}_{F_1}) \cong \Z[\pi_1(X_T)]$ and $C_k(\wt{X}_T^{CW},\wt{X}_{F_1})=0$ for $k \neq 2$.
By tensoring with~$R^m$ we have that
\[C_k^{\Phi}(X_T^{CW},X_{F_1};R) \cong C_k(\wt{X}_T^{CW},\wt{X}_{F_1}) \otimes_{\Z[\pi_1(X_T)]} R^m \]
is isomorphic to $R^m$ for $k=2$ and is zero otherwise.
Since $C_*^{\Phi}(X_T,X_{F_1};R) \simeq C_*^{\Phi}(X_T^{CW},X_{F_1};R)$, we therefore obtain that $H_k^\Phi(X_T, X_{F_1}; R)=0$ for $k \neq 2$ and $H_2^\Phi(X_T, X_{F_1}; R) \cong R^m$.

Since dually $X_T$ is obtained from $X_{F_2} \times I$ by attaching a single 5-dimensional 3-handle, we have that $H_k^\Phi(X_T,  X_{F_2};R)=0$ for $k \neq 3$.
For $j=1,2$ the long exact sequence in twisted homology with $R$-coefficients corresponding to the triple $(X_T, X_{F_j}, X_K)$ is
\begin{align} \label{eqn:twistedrelhomology}
\dots \to H_3^\Phi(X_T, X_{F_j}) \to H_2^\Phi(X_{F_j}, X_K) \xrightarrow{g_j} H_2^\Phi(X_T, X_K) \xrightarrow{h_j} H_2^\Phi(X_T, X_{F_j})\to \dots
\end{align}
and so we see that $g_2$ is surjective.
%Also $g_1$ is injective, but we never use this!

Now consider the following diagram, which is commutative since all maps are induced by various inclusions and natural long exact sequences.  The horizontal sequences come from long exact sequences of various pairs and all homology is appropriately twisted with coefficients in $R$.
\[
\begin{tikzcd}
H_2^\Phi(X_{F_1})\arrow{d} \arrow{r} & H_2^\Phi(X_{F_1}, X_K) \arrow{d}{g_1} \arrow{dr}{\partial_1} & & H_1^\Phi(X_{F_1}) \arrow{d} \\
H_2^\Phi(X_T) \arrow{r} & H_2^\Phi(X_T, X_K) \arrow{r}{\partial_T} &H_1^\phi(X_K)\arrow{dr}{j_2} \arrow{ur}{j_1} \arrow{r}{j_T} & H_1^\Phi(X_T) \\
H_2^\Phi(X_{F_2})  \arrow{u} \arrow{r} & H_2^\Phi(X_{F_2}, X_K) \arrow{u}{g_2} \arrow{ur}{\partial_2} & & H_1^\Phi(X_{F_2})\arrow{u}
\end{tikzcd}
\]
Since $g_2$ is surjective, we have that $P_2= \ker(j_2)= \Imm(\partial_2) = \Imm(\partial_T)$. Also,
\[P_1= \ker(j_1)= \Imm(\partial_1)= \Imm(\partial_T \circ g_1) \subseteq \Imm(\partial_T) = P_2.\]
So we have established the first conclusion of this proposition.

To establish the second conclusion, we recall from above that $H_2^{\Phi}(X_T, X_{F_1};R) \cong R^m$ has $R$-generating rank~$m$. Considering the long exact sequence of Equation~(\ref{eqn:twistedrelhomology}), we see that
\begin{align*}
\coker(g_1) = H_2^{\Phi}(X_T, X_K)/ \Imm(g_1)= H_2^{\Phi}(X_T, X_K)/ \ker(h_1) \cong \Imm(h_1)  \subseteq H_2^{\Phi}(X_T, X_{F_1})
\end{align*}
and so $\coker(g_1)$ has generating rank no more than $m$ as an $R$-module, by Lemma~\ref{lemma:gen-rank-facts}~(\ref{item:gen-rank-subgroup}).
We can therefore let $\{a_i\}_{i=1}^m$ be elements of $H_2(X_T, X_K)$ which represent generators of $\coker(g_1)$. Hence together with $\Imm(g_1)$  the $\{a_i\}_{i=1}^m$ generate $H_2(X_T, X_K)$ as an $R$-module.  Therefore $\partial_T(\Imm(g_1) \cup \{a_i\}_{i=1}^m)$ generates $\Imm \partial_T = P_2$. It follows that
\begin{align*}
P_1 \cup \{\partial_T(a_i)\}_{i=1}^m & =
 \Imm(\partial_1) \cup  \{ \partial_T(a_i)\}_{i=1}^m\\
&= \Imm(\partial_T \circ g_1) \cup  \{ \partial_T(a_i)\}_{i=1}^m \\
&= \partial_T(\Imm(g_1) \cup \{ a_i\}_{i=1}^m)
\end{align*}
generates $\Imm(\partial_T)=P_2$ as an $R$-module, and so we can let $x_i= \partial_T(a_i)$ for $i=1, \dots, m$.
\end{proof}

%From Proposition~\ref{prop:handleadditiondoublecover} we see that given a map $\chi \colon H_1(X_K^2) \to \Z_n$ we have that $\chi$ extends over $H_1(X^2_T)$ (and hence to $H_1(X^2_{F_2})$) if and only if $\chi$ vanishes on $\ker(j_T)= \Imm(\partial_T)= \Imm(\partial_2)$.
%In particular, if we already know that $\chi$  extends over $H_1(X^2_{F_1})$, then the only additional criterion is that $\chi(x)=0$, where $x$ is as in the statement of Proposition~\ref{prop:handleadditiondoublecover}.
%Presumably $x$ is represented by the lift of $\gamma$ to the double cover, as in the proof of the proposition.

\begin{prop}\label{prop:showing-d_2-large}
  Let $\Delta_1$ and $\Delta_2$ be slice discs for a knot $K$.  Let $P_j:=\ker(\A_{\Q}(K) \to \A_{\Q}(\Delta_j))$ for $j=1,2$.
  Suppose that $\gr(P_1)=\gr(P_2)=n$ and that $\gr(P_1 \cap P_2)=k$. Then $d_2(\Delta_1,\Delta_2) \geq n-k$.
\end{prop}

\begin{proof}
Suppose that $F$ is a genus $g$ surface to which both $\Delta_1$ and $\Delta_2$ stabilize by $g$ 1-handle additions and some number of 2-knot additions.
 We will show that $g \geq n-k$.
 By Proposition~\ref{prop:2spheresfirst}, for $j=1,2$ there exist a disc $\Delta_j'$ obtained from $\Delta_j$ by connected sum with some number of knotted 2-spheres such that $F$ is obtained from $\Delta_j'$ by $g$ 1-handle additions.
It follows from Proposition~\ref{prop:2spheresum} that for $j=1,2$ we have
 \[P_j' :=\ker(\A_{\Q}(K) \to \A_{\Q}(\Delta_i'))=P_j.
 \]

Let $P:= \ker(\A_{\Q}(K) \to \A_{\Q}(F))$.
 By Proposition~\ref{prop:kerneloftwistedH1}, we see that both $P_1'$ and $P_2'$ are submodules of $P$. We now argue  that the generating rank of $P$, considered as a $\Q[t^{\pm1}]$-module,  is at least $2n-k$. To see this we show that $\Imm(P_1'\oplus P_2' \to P)$ has generating rank at least $2n-k$ and apply Lemma~\ref{lemma:gen-rank-facts}~(\ref{item:gen-rank-subgroup}).
 Let $i_1 \colon P_1' \to P$ and $i_2 \colon P_2' \to P$ be the inclusion maps.  Both $P_1'$ and $P_2'$ are submodules of $P$, so
 \[\ker(i_1 \oplus -i_2 \colon P_1' \oplus P_2' \to P) = \{(p_1,p_2) \in P_1' \oplus P_2' \mid i_1(p_1) = i_2(p_2) \in P\} \cong P_1' \cap P_2'.\]
 We obtain a short exact sequence
 \[0 \to P_1' \cap P_2' \to P_1' \oplus P_2' \to \Imm(i_1 \oplus -i_2) \to 0,\]
and conclude by  Lemma~\ref{lemma:gen-rank-facts}~(\ref{item:gen-rank-ses}) that $\gr(\Imm(i_1 \oplus -i_2)) \geq 2n-k$. Therefore by Lemma~\ref{lemma:gen-rank-facts}~\eqref{item:gen-rank-subgroup}, $\gr(P) \geq 2n-k$.  Note that this uses that $\Q[t^{\pm1}]$ is a PID.

 However, Proposition~\ref{prop:kerneloftwistedH1} applied with $m=1$ also tells us that there exist some $x_1, \dots, x_g$ in $P$ such that $P$ is generated by $P_1' \cup \{x_1, \dots, x_g\}$. Therefore  the generating rank  of $P$ is at most $n+g$, and so we have $n + g \geq \gr(P) \geq 2n-k$, from which it follows as desired that~$g \geq n-k$.
%
% Old proof.
%
%Now suppose $F$ is a genus $g$ surface to which both $\Delta_1$ and $\Delta_2$ stabilize. We will show that $g \geq n$.
%Let $P:= \ker(\A_{\Q}(K) \to \A_{\Q}(F))$.
% By Proposition~\ref{prop:kerneloftwistedH1}, we see that both $P_1$ and $P_2$ are submodules of $P$ and so  the generating rank of $P$, considered as a $\Q[t^{\pm1}]$-module,  is at least $2n$.
% However, Proposition~\ref{prop:kerneloftwistedH1} (applied $g$ times and with $m=1$) also tells us that there exist some $x_1, \dots, x_g$ in $P$ such that $P$ is generated by $P_1 \cup \{x_1, \dots, x_g\}$. Therefore  the generating rank  of $P$ is at most $n+g$, and so we have as desired that $g \geq n$.
\end{proof}

The next proposition completes the proof of Theorem~\ref{thm:B}.

\begin{prop} Let $K_0$ be the knot $9_{46}$ and let $K=\#_{i=1}^n K_0$. Let $\Delta_1= \natural_{i=1}^n D_1$ and let $\Delta_2:=\natural_{i=1}^n D_2$ be the  `left band only' and `right band only' slice discs. Then
\[ d_2(\Delta_1, \Delta_2)=n.
\]
\end{prop}

\begin{proof}
First, note that  we can obtain both $\Delta_1$ and $\Delta_2$ from surgery on a genus $n$ Seifert surface for $K$ and so $d_2(\Delta_1, \Delta_2) \leq n$.

There is an identification
\[\A_{\Q}(K)  \cong \bigoplus_{i=1}^n \A_{\Q}(K_0) \cong \bigoplus_{i=1}^n \left( \Q[t^{\pm1}]/ \langle 2t-1 \rangle \oplus  \Q[t^{\pm1}]/ \langle t-2 \rangle   \right)
\]
 such that
\begin{align*}
P_1&:=\ker(\A_{\Q}(K) \to \A_{\Q}(\Delta_1)) = \bigoplus_{i=1}^n \Q[t^{\pm1}]/ \langle t-2 \rangle   \\
\text{ and } P_2&:=\ker(\A_{\Q}(K) \to \A_{\Q}(\Delta_2)) = \bigoplus_{i=1}^n \Q[t^{\pm1}]/ \langle 2t-1 \rangle .\\
\end{align*}
In particular, $P_1 \cap P_2 =\{0\}$.
Now, $\gr (P_1) = \gr (P_2) = n$, and $\gr(P_1 \cap P_2) =0$. It follows from Proposition~\ref{prop:showing-d_2-large} that $d_2(\Delta_1, \Delta_2)\geq n$ as required.
\end{proof}

\section{Secondary lower bounds using metabelian twisted homology}\label{section:slice-discs-2}

We now construct subtler examples of pairs of slice discs with high stabilization distance.

%
%\begin{thm}\label{thm:subtlerlargedistance}
% For every $n \in \mathbb{N}$ there exists a  slice knot $K_n$ with slice discs $D$ and $D'$ for $K_n$ such that
%\[\ker(\A(K_n) \to \A(D))=\ker(\A(K_n)\to \A(D'))\] and yet $d_2(D, D') \geq n$.
%\end{thm}

\subsection{Satellite knots and satellite slice discs}

Our examples come from the satellite construction.
Let $R$ and $J$ be knots and let $\eta \subset S^3 \ssm R$ be an unknotted simple closed curve in the complement of $R$.
Recall that $S^3 \ssm \nu(\eta) \cup X_J \cong S^3$, where the meridian of $\eta$ is identified with the longitude of $J$, and vice versa. The image of $R \subset S^3 \ssm \nu (\eta)$ under this homeomorphism is by definition the satellite knot $R_{\eta}(J)$.

It is a well known fact that if $R$ and $J$ are slice knots and $\eta$ is any unknot in the complement of $R$, then the satellite knot $R_{\eta}(J)$ is also slice. It will be useful to have an explicit construction of a slice disc $\Delta_D$ for $R_{\eta}(J)$ coming from a choice of slice discs $\Delta_0$ for $R$ and $D$ for $J$, together with compatible degree 1 maps $f \colon X_{R_\eta(J)} \to X_R$ and $g \colon X_{\Delta_D} \to X_{\Delta_0}$.

\begin{cons}[Satellite slice discs and degree 1 maps]\label{cons:satellitedisc}
Let $R$ be a knot with slice disc  $\Delta_0$ and let $\eta$ be an unknotted curve in $S^3 \ssm \nu(R)$.
Identify $D^4\supset \Delta_0$ as $D^2 \times D^2$ in such a way that when we consider $\partial(D^2 \times D^2) = (S^1 \times D^2) \cup (D^2 \times S^1)$ we have $D^2 \times S^1= \nu(\eta)$ and so $R=\partial \Delta_0 \subseteq S^1 \times D^2$.

Now let $J$ be a knot with slice disc $D$. We obtain a slice disc denoted $\Delta_D$ for $R_{\eta}(J)$ by considering
\[ \Delta_0\subseteq  D \times D^2= \nu(D) \subset D^4.\]

Note that $X_{\Delta_D}= X_{\Delta_0} \cup_{S^1 \times D^2} X_D$, where $S^1 \times D^2$ is identified with $\nu(\eta) \subseteq X_R \subset \partial X_{\Delta_0}$ and with $S^1 \times D \subset \partial X_D$, and that this identification is evidently compatible with the decomposition $X_{R_{\eta}(J)}= (X_{R} \ssm \nu(\eta)) \cup_{T^2} X_J.$

For every knot $J$ there is a standard degree 1 map $f_0 \colon X_J \to X_U$ which sends $\mu_J$ to $\mu_U$ and $\lambda_J$ to $\lambda_U$,  and for any slice disc $D$ there is a similar degree one map $g_0 \colon X_D \to X_E$, where $E$ denotes the standard slice disc for the unknot.
For the sake of completeness, we give this construction, emphasizing that one can choose $g_0$ to be an extension of $f_0$.

Parametrize
\[ \nu\left( \partial X_J\right)= \partial X_J \times [0,\delta] =\{ (p,s,t)\in S^1 \times \left( [0,2\pi]/\sim \right) \times [0,\delta]\},
\]
where $\{(p,0,0)\}= \lambda_J$ and $\{(1, s, 0)\}= \mu_J$.
Now let $F \subset X_J$ be a (truncated) Seifert surface for $J$ with tubular neighborhood $\nu(F)= F \times [0,\varepsilon]$. We can assume that
\[ \nu(F) \cap \nu(\partial X_J)=  \{ (p,s,t)\in S^1 \times [0,\varepsilon] \times [0,\delta]\},
\]
as illustrated below.
\begin{figure}[h]
\includegraphics[height=3cm]{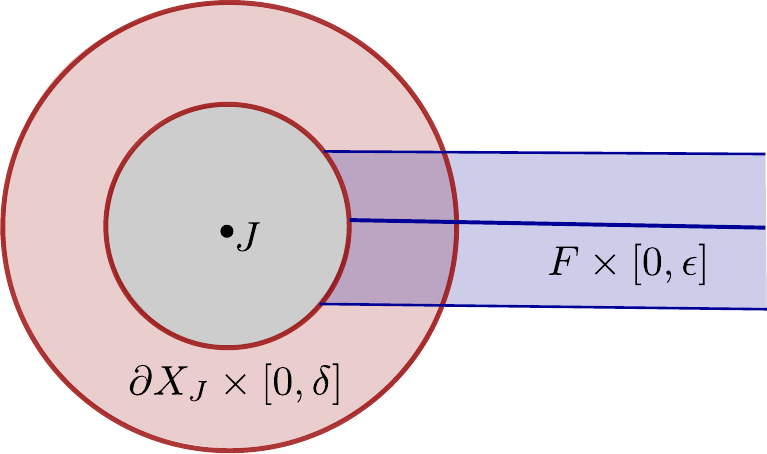}
\caption{A cross section of $X_J$ near its boundary.  Note that the grey region represents $\nu(J)$ and is therefore not part of $X_J$.}
\end{figure}
%\footnote{Can you change $\epsilon$ to $\varepsilon$ in this figure? I feel like the referee and I would get on...}

We  write $X_U= S \times D$ for $S= \left([0,\varepsilon]/ \sim\right) \cong S^1$ and $D= (S^1 \times [0,\delta])/ (S^1 \times {\delta}) \cong D^2$.
Define $f_0$ on $\nu(\partial X_J)$ by
\begin{align*}
f_0(p,s,t)= \begin{cases}
  (s,(p,t)) & \text{ if }  0 \leq s \leq \varepsilon \\
(\varepsilon,(p,t)) & \text{ if } \varepsilon<s,
\end{cases}
\end{align*}
and then extend over the rest of $\nu(F)= F \times [0,\varepsilon]$ by
$f_0(y, s)=(s,(0,\delta))$. Finally, for any $x$ in neither $\nu(F)$ nor $\nu(\partial X_K)$, we define $f_0(x)= (\varepsilon, (0, \delta))$.

The construction of $g_0$ is very similar, only with a compact orientable 3-manifold $G$ with boundary $\partial G= F \cup_J D$ playing the r{o}le
of the Seifert surface: we extend $f_0$ as defined above on $X_J$  over $X_J \times I$, then over the rest of $\nu(\partial X_D)$, then over $\nu(G) \cong G \times I$ and then send the entirety of $X_D \setminus (\nu(\partial X_D) \cup \nu(G))$ to a single point in $X_E$.

Here are the details, which closely parallel the construction of~$f_0$, though with extra care  taken to ensure that $g_0|_{X_J} = f_0$:

First parametrize a neighborhood of the slice disc $D$ as $D^2 \times D^2$, naturally a manifold with corners, such that $S^1 \times D^2$ is a tubular neighborhood of $J$ and $S^1 \times S^1 = \partial X_J$. Consider a collar on this part of $\partial X_D$ as follows. We think of $X_D$ as a manifold with corners, with $\partial X_J$ the corner set, dividing $\partial X_D$ as $X_J \cup_{\partial X_J} D^2 \times S^1$. Then we consider a collar on the $D^2 \times S^1$ part of the boundary that restricts on $X_J$ to a collar for $\partial X_J$ in $X_J$.
Parametrize this collar as
\[ \nu\left( D^2 \times S^1 \right)= D^2 \times S^1 \times [0,\delta] =\{ (p,s,t)\in D^2 \times \left( [0,2\pi]/\sim \right) \times [0,\delta]\},
\]
where $\{(p,0,0)\}$ is a push-off of the slice disc with boundary $\lambda_J$  and $\{(1, s, 0)\}= \mu_J$.

Now let $G \subset X_D$ be a (truncated) 3-manifold with $\partial G = F \cup \{(p,0,0)\}$, with tubular neighborhood $\nu(G)= G \times [0,\varepsilon]$.
We note that the existence of such a 3-manifold follows from a standard obstruction theoretic argument, see e.g.~\cite[Lemma~8.14]{Lickorish-text}.
 We can assume this restricts to the tubular neighborhood of $F$ used above in the definition of $f_0$, and that
\[ \nu(G) \cap \nu(D^2 \times S^1)=  \{ (p,s,t)\in D^2 \times [0,\varepsilon] \times [0,\delta]\}.
\]

We  write $X_E= S \times B$ for $S= \left([0,\varepsilon]/ \sim\right) \cong S^1$ and $B= (D^2 \times [0,\delta])/ (D^2 \times {\delta}) \cong D^3$.
Note that we have a natural inclusion $D \subset B$ corresponding to $X_U = S \times D \subset S \times B = X_E$.
Define $g_0$ on $\nu(D^2 \times S^1)$ by
\begin{align*}
g_0(p,s,t)=
\begin{cases}
  (s,(p,t)) & \text{ if }  0 \leq s \leq \varepsilon \\
(\varepsilon,(p,t)) & \text{ if } \varepsilon<s,
\end{cases}
\end{align*}
and then extend over the rest of $\nu(G)= G \times [0,\varepsilon]$ by
$g_0(y, s)=(s,(0,\delta))$. Finally, for any $x$ in neither $\nu(G)$ nor $\nu(D^2 \times S^1)$, we define $g_0(x)= (\varepsilon, (0, \delta))$.

By using the above decompositions $X_{R_{\eta}(J)}= (X_{R} \ssm \nu(\eta)) \cup_{T^2} X_J$
 and $X_{\Delta_D}= X_{\Delta_0} \cup_{S^1 \times D^2} X_D$, we obtain compatible degree 1 maps
\begin{align*}
f= \Id \cup f_0 \colon  X_{R_{\eta}(J)} \to X_R \text{ and }  g=\Id \cup g_0 \colon X_{\Delta_D} \to X_{\Delta_0}.
\end{align*}
This completes Construction~\ref{cons:satellitedisc}.
\end{cons}

Recall that for a connected space $X$ equipped with a surjective map $\varepsilon \colon \pi_1(X) \to \Z$, we let $\A(X)$ denote the induced $\Z[t^{\pm1}]$-twisted first homology, and for a knot or disc $L $ we often let $\mathcal{A}(L)$ denote $\mathcal{A}(X_L)$.
%We remark that $\A_{\Q}(L) \cong \A(L) \otimes \Q$, because $\Q$ is flat as a $\Z$-module.\footnote{AM: Also maybe move this to twisted homology section? MP: I don't really mind, if we use it here then we can keep it here?}

\begin{prop}\label{prop:homologycoverinfection}
Let $R$, $\Delta_0$, $\eta$, $J$, and $D$ be as above. Suppose that the linking number of $\eta$ and $R$ in $S^3$ is 0.
Letting $f$ and $g$ be the degree 1 maps discussed above, the following diagram commutes, where the horizontal maps are the usual inclusion induced maps:
\[
\begin{tikzcd}
&\A(R_{\eta}(J)) \arrow{d}{f_*} \arrow{r} &\A(\Delta_D) \arrow{d}{g_*}& \\
&\A(R) \arrow{r} &\A(\Delta_0).&
 \end{tikzcd}
\]
Moreover, $f_*$ and $g_*$ are isomorphisms and so
\[\ker(\A(R_{\eta}(J)) \to  \A(\Delta_D))= f_*^{-1}( \ker(\A(R) \to \A(\Delta_0))) \cong \ker(\A(R) \to \A(\Delta_0))\]
is independent of the choice of slice disc $D$ for $J$.
\end{prop}

\begin{proof}
The fact that the diagram commutes follows immediately from the compatibility of $f$ and $g$ as defined in Construction~\ref{cons:satellitedisc}. Since the linking number of $R$ and $\eta$ is 0, the fact that
$f_*$ is an isomorphism is a standard fact (one can also imitate the proof of Proposition~\ref{prop:satellitekernel} in a simpler setting). Briefly, one compares the Mayer-Vietoris sequences for $X_{R_{\eta}(J)} = X_{R\cup \eta} \cup_{S^1 \times S^1} X_J$ and $X_R = X_{R_{\eta}(U)} = X_{R\cup \eta} \cup_{S^1 \times S^1} X_U$. The fact that the winding number of $\eta$ is zero implies that the induced representations $\pi_1(X_J) \to \Z$ and $\pi_1(X_U) \to \Z$ are trivial, so  $H_1(X_J;\Z[t^{\pm 1}])\cong H_1(X_U;\Z[t^{\pm 1}]) \cong \Z[t^{\pm 1}]$.
% \footnote{Added a brief explanation.}

To see that $g_*$ induces an isomorphism consider the following diagram, where the rows are the Mayer-Vietoris sequences in $\Z[t^{\pm 1}]$-coefficients corresponding to the decompositions $X_{\Delta_D}= X_{\Delta_0} \cup_{S^1 \times D^2} X_D$ and $X_{\Delta_0}= X_{\Delta_0} \cup_{S^1 \times D^2} X_E$. We have replaced the $H_0$ terms with zeroes, since the maps from $H_0(S^1 \times D^2;\Z[t^{\pm 1}])$ are injective.
\[
\adjustbox{scale=0.9}{%
\begin{tikzcd}
  H_1(S^1 \times D^2;\Z[t^{\pm 1}]) \arrow{d} \arrow{r} & H_1(X_{\Delta_0};\Z[t^{\pm 1}]) \oplus H_1(X_D;\Z[t^{\pm 1}]) \arrow{d}{\operatorname{Id} \oplus (g_0)_*} \arrow{r} & H_1(X_{\Delta_D};\Z[t^{\pm 1}]) \arrow{r} \arrow{d}{g_*} & 0  \\
  H_1(S^1 \times D^2;\Z[t^{\pm 1}])  \arrow{r} & H_1(X_{\Delta_0};\Z[t^{\pm 1}]) \oplus H_1(X_E;\Z[t^{\pm 1}])  \arrow{r} & H_1(X_{\Delta_0};\Z[t^{\pm 1}]) \arrow{r} & 0
\end{tikzcd}}\]
Since the linking number of $\eta$ and $R$ is 0, the cores of the copies of $S^1 \times D^2$ along which the spaces are glued, when  thought of as fundamental group elements, map trivially to $\Z$ via the appropriate version of $\varepsilon$. Therefore $H_1(S^1 \times D^2;\Z[t^{\pm 1}]) \cong H_1(S^1 \times D^2;\Z) \otimes \Z[t^{\pm 1}] \cong \Z[t^{\pm 1}].$  Similarly, since $S^1 \times D^2\to X_D$ and $S^1 \times D^2 \to X_E$ are $\Z$-homology equivalences, the maps  $\pi_1(X_D) \to \Z$ and $\pi_1(X_E) \to \Z$ are likewise trivial, and so the maps $H_1(S^1 \times D^2;\Z[t^{\pm 1}]) \to H_1(X_D;\Z[t^{\pm 1}])$ and $H_1(S^1 \times D^2;\Z[t^{\pm 1}]) \to H_1(X_E;\Z[t^{\pm 1}])$ are isomorphisms.  It follows that the diagram above reduces to the diagram:
\[\begin{tikzcd}
& H_1(X_{\Delta_0};\Z[t^{\pm 1}]) \arrow{d}{\operatorname{Id}} \arrow{r}{\cong} & H_1(X_{\Delta_D};\Z[t^{\pm 1}]) = \A(\Delta_D) \arrow{d}{g_*} & \\
  & H_1(X_{\Delta_0};\Z[t^{\pm 1}]) \arrow{r}{\cong} & H_1(X_{\Delta_0};\Z[t^{\pm 1}])= \A(\Delta_0). &
\end{tikzcd}\]
Therefore the right hand vertical map is an isomorphism induced by $g$, as required.  %Since $g$ restricts on the boundary to $f$, the square commutes.
\end{proof}

\begin{exl}\label{exl:61}
Let $R$ be the slice knot $6_1$, with unknotted curve $\eta \in S^3 \ssm \nu(R)$ as shown on the left of Figure~\ref{fig:61}.
We will be interested in the satellite knot $R_{\eta}(J)$, depicted on the right of Figure~\ref{fig:61}, for certain choices of $J$.
\begin{figure}[ht]
  \includegraphics[height=5cm]{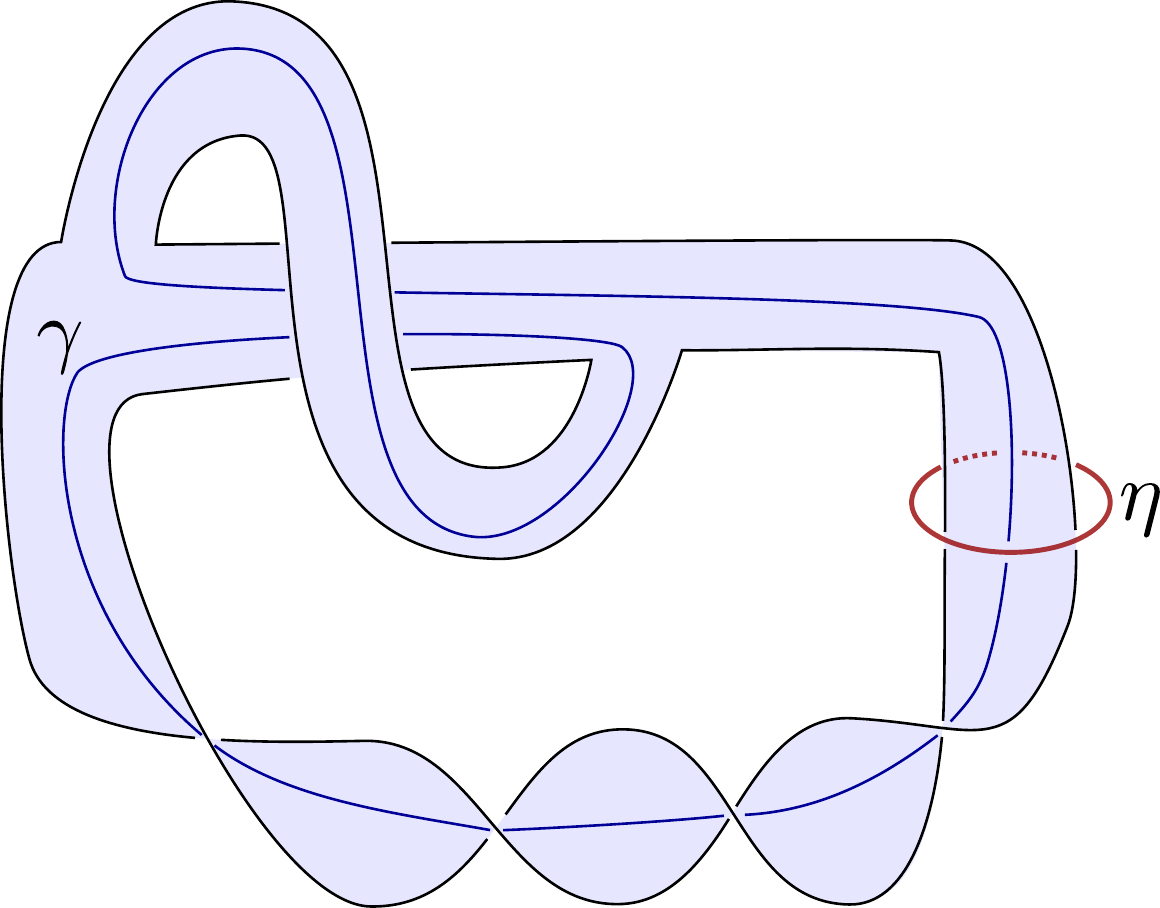}
  \qquad \qquad
    \includegraphics[height=5cm]{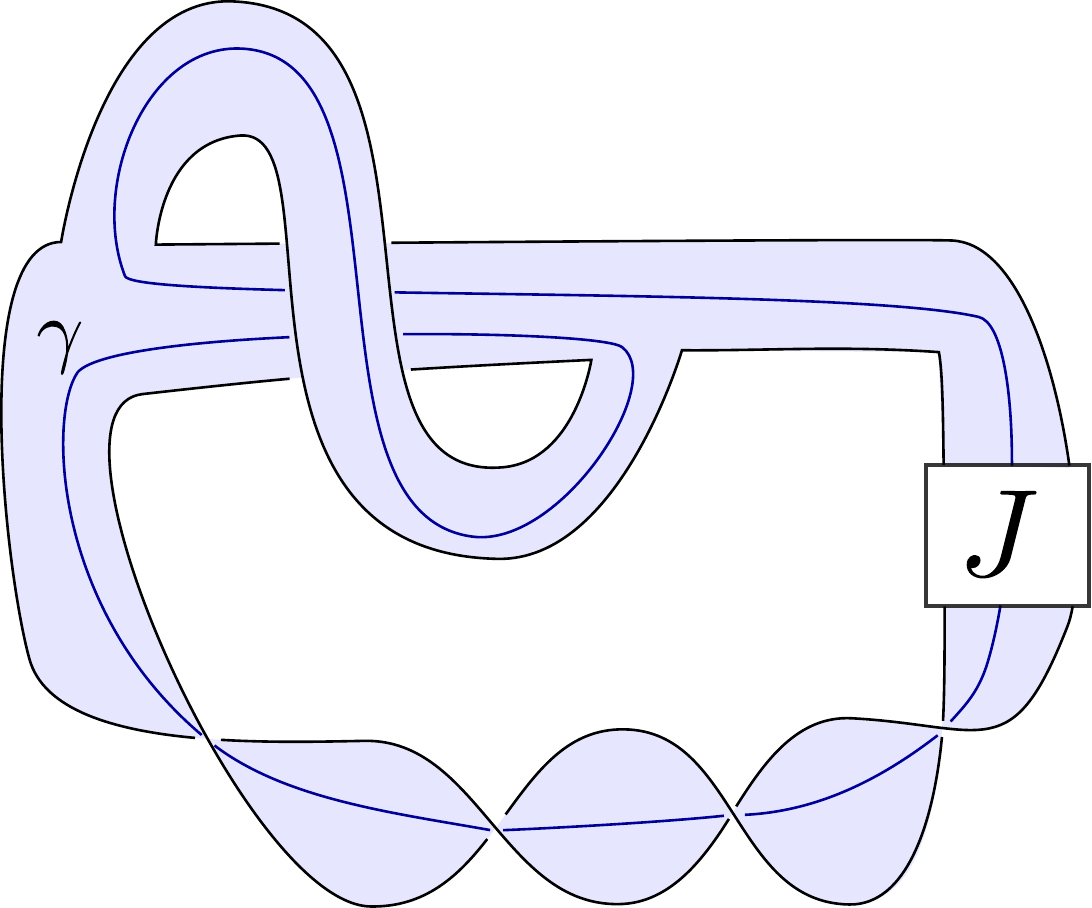}
  \caption{The knot $R=6_1$ with a genus 1 Seifert surface $F$, a $0$-framed curve $\gamma$ on $F$, and an infection curve $\eta$ (left) and the satellite knot $R_{\eta}(J)$ (right).}
  \label{fig:61}
\end{figure}
Note that $\eta$ does not intersect  $F$ and so~$R_{\eta}(J)$ has a genus 1 Seifert surface $F_J$  as shown on the right of Figure~\ref{fig:61}.
The illustrated homologically essential 0-framed curve on $F_J$ (that, in a mild abuse of notation, we also call  $\gamma$) is  isotopic to the knot $J$ when thought as a curve in $S^3$.

Let $\Delta_0$ denote the standard slice disc for $R$, obtained by surgering $F$ along $\gamma$. Given a slice disc $D$ for $J$, in Construction~\ref{cons:satellitedisc}  we built a slice disc $\Delta_D$  for $R_{\eta}(J)$.  In this context, one can interpret this construction as follows.  Push the interior of $F_J$ into the interior of $D^4$, then remove a small neighborhood of $\gamma$ in $F_J$. This creates two new boundary components, which may be capped off with parallel copies of $D$ to yield  $\Delta_D$.  We note that a single 1-handle attachment to $\Delta_D$ that connects the two parallel copies of $D$ returns the (pushed in) Seifert surface $F_J$, and so if $D$ and $D'$ are two different slice discs for $J$ we always have that $d_2(\Delta_D, \Delta_{D'}) \leq 1$, even if  $d_2(D, D')$ is large.

As in Example~\ref{exl:946}, we can pick a basis for the first homology of the Seifert surface $F$ for which the Seifert matrix is given by
\[A=\left[ \begin{array}{cc} 1 & 1 \\ 0 & -2 \end{array}\right]\] and manipulate $tA-A^T$ to see that $\A(R) \cong \Z[t^{\pm1}]/ \langle (2t-1)(t-2) \rangle$.  We have that $\A(\Delta_0) \cong \Z[t^{\pm1}]/ \langle 2t-1 \rangle$, and that the kernel of the inclusion induced map $\A(R) \to \A(\Delta_0)$ is exactly $(t-2) \A(R)$. Details can be found in e.g.\ \cite[Section~5.2]{Conway-Powell}.
%\footnote{MP: added another shameless self citation.}
 Additionally, by substituting $t=-1$ into the above computations we discover the homology of the 2-fold branched covers: $H_1(\Sigma_2(R)) \cong \Z_9$ and $\ker(H_1(\Sigma_2(R);\Z) \to H_1(\Sigma_2(D^4, \Delta_0);\Z))= 3 \Z_9$.
\end{exl}

\subsection{Metabelian twisted homology}

We will use twisted homology coming from metabelian representations that factor through the dihedral group $D_{2n} \cong \Z_2 \ltimes \Z_n$. As noted in the introduction, these representations originate in the work of Casson-Gordon~\cite{Casson-Gordon:1978-1,Casson-Gordon:1986-1}. Our perspective on these representations is particularly indebted to the work of \cite{HKL08}, as well as \cite{Kirk-Livingston:1999-2, Let00, Friedl:2003-4}.

\begin{cons}
Consider a knot $K$ with preferred meridian $\mu_0$, an abelianization map $\varepsilon \colon \pi_1(X_K) \to \Z$, and  a map $\psi\colon H_1(X_K^2) \to \Z_n$ for some prime $n$, where $X^2_K$ is the 2-fold cyclic cover of~$X_K$.  Assume that the map $\psi$ factors as \[\psi \colon H_1(X_K^2) \to H_1(\Sg_2(K)) \xrightarrow{\chi} \Z_n,\] where the first map is induced by the inclusion $X_K^2 \subset \Sg_2(K)$, so that $\psi$ is determined by~$\chi$.  Define
\begin{align*}
\phi_\chi \colon \pi_1(X_K) &\to \Z_2 \ltimes \Z_n \text{ by }
\phi_\chi(\gamma)=([\varepsilon(\gamma)], \psi(\mu_0^{-\varepsilon(\gamma)}\gamma)),
\end{align*}
noting that $\mu_0^{-\varepsilon(\gamma)}\gamma \in \ker (\pi_1(X_K) \to \Z_2)$ and so represents an element in $\pi_1(X^2_K)$.
Letting  $\xi_n= e^{2 \pi i/n}$, we have a standard map
\begin{align*}
 \alpha \colon \Z_2 & \ltimes \Z_n \to \GL_2(\Z[\xi_n])\\
(a, b) &\mapsto
\left[ \begin{array}{cc}
0 & 1 \\
1 & 0
\end{array} \right]^{a}
\left[ \begin{array}{cc}
\xi_n^{b} & 0 \\
0 & \xi_n^{-b}
\end{array} \right].
\end{align*}

In particular, we obtain a representation $\alpha_\chi= \alpha \circ \phi_\chi$ of $\pi_1(X_K)$ into $\GL_2(\Z[\xi_n])$. We will be interested in the corresponding twisted homology $H_*^{\alpha_\chi}(X_K, \Z[\xi_n])$, especially when $\Z[\xi_n]$ is a PID, e.g.\ when $n=3$  and $\Z[\xi_3]$ is the ring of \emph{Eisenstein integers}.
For a connected space $X$ together with a map $\phi \colon \pi_1(X) \to \Z_2 \ltimes \Z_n$, we will sometimes let $H_*^{\phi}(X;\Z[\xi_n])$ be shorthand for $H_*^{\alpha \circ \phi}(X; \Z[\xi_n])$. When the coefficients are clearly understood and we are short of space, we shall abbreviate this still further to $H_*^{\phi}(X)$.
\end{cons}

\begin{remark}\label{rem:h0}
We will often have two compact connected spaces $X \subset Y$ and a map  $\alpha_\psi= \alpha \circ \phi_\psi \colon \pi_1(Y) \to \GL_2(\Z[\xi_n])$ arising as above
from $\varepsilon \colon \pi_1(Y) \to \Z$ and $\psi  \colon Y^2 \to \Z_n$.
We wish to consider the inclusion induced maps
\[i_k \colon H_k^{\alpha_{\psi} \circ i_*} (X, \Z[\xi_n]) \to H_k^{\alpha_{\psi}} (Y, \Z[\xi_n]).\]
To understand this map when $k=0$, pick a CW structure on $X$ with a single 0-cell $x$ and 1-cells $g_1, \dots, g_m$ and extend it to a CW structure on $Y$ by first adding 1-cells $g_{m+1}, \dots, g_{m+m'}$. Of course, there may be many additional $n$-cells for $n \geq 2$, but these will not impact  $H_0$ computations. The relevant twisted cellular chain complexes are
\[ C_0^{\alpha_{\psi}\circ i_*}(X) \cong C_0^{\alpha_{\psi}}(Y) \cong \Z[\xi_n]^2,\,
C_1^{\alpha_{\psi}\circ i_*}(X)\cong \Z[\xi_n]^{2m}, \text{ and }C_1^{\alpha_{\psi}}(Y)\cong \Z[\xi_n]^{2(m+m')}
\]
with differential maps  given by the matrices
\begin{align*}
d_1^X=& \left[ \begin{array}{cccc} [\alpha_{\psi}(g_1) - \Id] & [\alpha_{\psi}(g_2) - \Id] & \dots & [\alpha_{\psi}(g_m) - \Id] \end{array} \right]
\\
d_1^Y=&\left[ \begin{array}{cccccc} [\alpha_{\psi}(g_1) - \Id] & [\alpha_{\psi}(g_2) - \Id] & \dots& [\alpha_{\psi}(g_m) - \Id] & \dots & [\alpha_{\psi}(g_{m+m'}) - \Id] \end{array} \right].
\end{align*}
It follows that the map $i_0$ is always a surjection, and is an isomorphism if and only if
\[
\Span\{ \Imm [\alpha_{\psi}(g_i) - \Id]\}_{i=1}^m= \Span\{ \Imm [\alpha_{\psi}(g_i) - \Id]\}_{i=1}^{m+m'}.
\]
In order to ensure that $i_0$ is an isomorphism, it therefore suffices to check that
the two maps $\phi_\psi \circ i_*$ and $\phi_\psi$ have the same image in $\Z_2 \ltimes \Z_n$. In the rest of this section, whenever we claim that $i_0$ is an isomorphism it will be because these two images agree, though in the interest of brevity we will often leave that verification to the reader.
\end{remark}

We will need a computation of the twisted homology of a knot complement with respect to certain abelian representations into $\GL_2(\Z[\xi_n])$.
It will be convenient to have the following notation.

\begin{notation}\label{notation:Axi}
Let $X$ be a connected space equipped with a surjection $\varepsilon \colon \pi_1(X) \twoheadrightarrow \Z$, and let $\xi$ be a root of unity.
Define
$\A_{\xi}(X):=\mathcal{A}(X) \otimes_{\Z[t^{\pm1}]} \Z[\xi],$
where $\Z[\xi]$ has the $\Z[t^{\pm1}]$-module structure induced by $t \cdot a:= \xi a$.

Also,  for any $\Z[\xi]$-module $M$, let $\widebar{M}$ denote the module with conjugate $\Z[\xi]$-structure
and let  $M^{1 \oplus \bar{1}}:= M \oplus \widebar{M}$.
\end{notation}

\begin{lem}\label{lem:computation1}
Let $X$ be a connected space with a surjection   $\varepsilon\colon \pi_1(X) \twoheadrightarrow \Z$, and define
$\phi \colon \pi_1(X)  \to \GL_2(\Z[\xi_n])$ by
\begin{align*}
\gamma &\mapsto  \left[ \begin{array}{cc}
 \xi_n^{\varepsilon(\gamma)} & 0 \\ 0 & \xi_n^{-\varepsilon(\gamma)}
  \end{array} \right].
  \end{align*}
% Suppose  that $H_*(X^\infty)=0$ for $* \geq 2$, where $X^{\infty} \to X$ is the $\varepsilon$-induced $\Z$-cover of $X$.
Then
$H_1^{\phi}(X; \Z[\xi_n]) \cong \A_{\xi}(X) \oplus \A_{\bar{\xi}}(X) \cong \A_{\xi}(X)^{1 \oplus \bar{1}}.$
\end{lem}
%\footnote{MP: added $\cong \A_{\xi}(X)^{1 \oplus \bar{1}}$.}

\begin{proof}
First, note that $H_1^{\phi}(X; \Z[\xi_n])\cong H_1^{\theta}(X; \Z[\xi_n])^{1 \oplus \bar{1}}$, where $\theta \colon \pi_1(X) \to \Z[\xi_n]^\times$ is given by $\theta(\gamma)= \xi_n^{\varepsilon(\gamma)}$.
So it suffices to show that $H_1^{\theta}(X; \Z[\xi_n]) \cong \A_{\xi}(X)$.

Let $X^{\infty} \to X$ be the $\varepsilon$-induced $\Z$-cover of $X$.
Note that $\theta(\gamma)=0$ if and only if $\varepsilon(\gamma) \equiv 0 \mod n$, and so the $\theta$-induced cover of $X$ is the $n$-fold cyclic cover $X^{n}$. We can compute
$H_1^{\theta}(X; \Z[\xi_n])$ as
\begin{align*}
H_1^{\theta}(X; \Z[\xi_n]) &= H_1\left( C_*(X^n) \otimes_{\Z[\Z_n]} \Z[\xi_n]\right) = H_1\left( C_*(X^\infty) \otimes_{\Z[t^{\pm1}]} \Z[\xi_n]\right).
\end{align*}
The K\"unneth spectral sequence~\cite[Theorem 5.6.4,~p.~143]{Weibel94} tells us that
since $C_*(X_{\infty})$ is a bounded below complex of flat (in fact free) $\Z[t^{\pm1}]$-modules, there is a boundedly converging upper right quadrant  spectral sequence:
\[E_{p,q}^2= \Tor_p^{\Z[t^{\pm1}]}(H_q(X^{\infty}), \Z[\xi_n]) \,\Rightarrow \, H_{p+q}(C_*(X^\infty) \otimes_{\Z[t^{\pm1}]} \Z[\xi_n]).
\]

The only $E^2_{p,q}$ which could potentially contribute to $H_{1}(C_*(X^\infty) \otimes_{\Z[t^{\pm1}]} \Z[\xi_n])$ are $(p,q)\in \{(1,0), (0,1)\}$. The only relevant differential could be $d^2_{2,0} \colon E^2_{2,0} \to E^2_{0,1}$. However.
\begin{align*}
E^2_{2,0}&=\Tor_2^{\Z[t^{\pm1}]}(H_0(X^{\infty}), \Z[\xi_n])= \Tor_2^{\Z[t^{\pm1}]}(\Z[t^{\pm1}]/ \langle t-1 \rangle, \Z[\xi_n]) \\ &=\Tor_2^{\Z[t^{\pm1}]}(\Z, \Z[\xi_n])=0,
\end{align*}
since as a $\Z[t^{\pm1}]$-module $\Z$ has a length 1 projective resolution. Therefore the spectral sequence collapses on the 1-line at the $E^2$ page, and it suffices to compute $E^2_{0,1}$ and $E^2_{1,0}$.
We have that
\begin{align*}
E^2_{1,0}&= \Tor_1^{\Z[t^{\pm1}]}(H_0(X^{\infty}), \Z[\xi_n]) \\
&= \Tor_1^{\Z[t^{\pm1}]}(\Z[t^{\pm1}]/ \langle t-1 \rangle, \Z[\xi_n]) \\
&\cong\{ x \in \Z[\xi_n] \mid (t-1)\cdot x=0\} \\
& \cong \{ x \in \Z[\xi_n] \mid (\xi_n-1)x=0\} =0.
\end{align*}
Finally, since
\[E^2_{0,1} = \Tor_0^{\Z[t^{\pm1}]}(H_1(X^{\infty}), \Z[\xi_n]) \cong H_1(X^{\infty}) \otimes_{\Z[t^{\pm1}]} \Z[\xi_n]= \A_\xi(X)\]
 we obtain our desired result.
\end{proof}
%
%\begin{lem}\label{lem:computation2}
%Given a space $X$ and a surjection $\varepsilon \colon \pi_1(X) \to \Z$,  define
%$\phi \colon \pi_1(X) \to \GL_2(\Z[\xi_n])$ by
% $\phi(\gamma)=
% \left( \begin{array}{cc}
%0& 1\\ 1&0  \end{array} \right)^{\varepsilon(\gamma)}$.
%  Then
%  \[ H_1^{\phi}(X; \Z[\xi_n]) \cong \]
%\end{lem}
%\begin{proof}
%Note that the $\phi$-induced cover of $X$ is the double cover induced by postcomposing $\varepsilon$ with $\Z \to \Z_2$.
%\end{proof}

Recall that given a slice knot $R$ with slice disc $\Delta_0$,  a slice knot $J$ with slice disc $D$, and an unknot $\eta$ in the complement of $R$, in Construction~\ref{cons:satellitedisc} we built degree one maps  $f \colon X_{R_{\eta}(J)} \to X_R$ and $g \colon X_{\Delta_D} \to X_{\Delta_0}$. The following proposition analyzes the $f$- and $g$-induced maps on certain twisted first homology modules under some additional conditions.

\begin{prop}\label{prop:satellitekernel}
Let $R$ be a slice knot with slice disc $\Delta_0$  and $J$ be a slice knot with slice disc $D$.
Let $\eta$ be an unknot in the complement of $R$ which generates $\A(R)$.
Suppose that $n$ is prime and $\chi \colon H_1(\Sg_2(R)) \to \Z_n$ is a nontrivial map such that $\phi_\chi$ extends to  $\Phi \colon \pi_1(X_{\Delta_0}) \to \Z_2 \ltimes \Z_{n}$.
There are identifications
\begin{align*}
 H_1^{\phi_\chi \circ f_*}(X_{R_{\eta}(J)}, \Z[\xi_{n}]) &\cong H_1^{\phi_\chi}(X_R, \Z[\xi_{n}]) \oplus \A_{\xi_n}(J)^{1 \oplus \bar{1}} \\
H_1^{\Phi \circ g_*}(X_{\Delta_D}, \Z[\xi_{n}]) & \cong H_1^{\Phi}(X_{\Delta_0}, \Z[\xi_{n}]) \oplus \A_{\xi_n}(D)^{1 \oplus \bar{1}}.
 \end{align*}
Moreover, these are natural with respect to inclusion maps; in particular
\[P:= \ker\left(H_1^{\phi_\chi \circ f_*} (X_{R_{\eta}(J)}, \Z[\xi_{n}])  \to  H_1^{\Phi \circ g_*}(X_{\Delta_D}, \Z[\xi_{n}]) \right)
\]
splits as the direct sum of the corresponding kernels $P_R \oplus P_J^{1 \oplus \bar{1}}$, where
\begin{align*}
P_R&:=\ker\left((H_1^{\phi_\chi}(X_R, \Z[\xi_{n}]) \to H_1^{ \Phi}(X_{\Delta_0}, \Z[\xi_{n}]) \right)\\
P_J^{1 \oplus \bar{1}} &:= \ker\left(\A_{\xi_n}(J)^{1 \oplus\bar{1}} \to \A_{\xi_n}(D)^{1 \oplus\bar{1}} \right)= \ker\left(\A_{\xi_n}(J) \to \A_{\xi_n}(D)\right)^{1 \oplus\bar{1}}.
 \end{align*}
\end{prop}

The proof of Proposition~\ref{prop:satellitekernel}, while somewhat long and notation heavy, essentially follows from careful consideration of the relationship between four Mayer-Vietoris long exact sequences.
These sequences are related by the maps induced from the following commutative diagram, where we remind the reader that horizontal maps are inclusions and vertical maps are defined as in Construction~\ref{cons:satellitedisc}:
\[ \begin{tikzcd}
X_{R_{\eta}(J)} =X_R \ssm \nu(\eta) \cup X_J \arrow[swap, "f= \Id \cup f_0"]{d} \arrow["i_\eta\, \cup\,  i_J"]{r} & X_{\Delta_0} \cup X_D=X_{\Delta_D} \arrow[" \Id \cup g_0=g"]{d}  \\
 X_R= X_R \ssm \nu(\eta) \cup X_U \arrow[swap, "i_\eta \, \cup\, i_U"]{r} & X_{\Delta_0} \cup  X_E=X_{\Delta_0}.\\
\end{tikzcd}
\]

\begin{proof}
We abbreviate $X_R \ssm \nu(\eta)$ by $X_R \ssm \eta$ and let $\xi= \xi_n= e^{2 \pi i/n}$.

Since $\eta \in \pi_1(X_R)^{(1)}$, when we restrict $(\alpha \circ \phi_\chi) \circ f_*$ to $\pi_1(X_J)$ we see that every element of $\pi_1(X_J)$ is sent to a matrix of the form \[\left[ \begin{array}{cc} \xi^b & 0 \\0  & \xi^{-b} \end{array} \right]\] for some $b \in \Z_{n}$.
In particular, this restriction factors through $H_1(X_J;\Z) \cong \Z$.
The fact that $\eta$ generates $\A(R)$ implies that the lifts of $\eta$
 to $X_R^2$ generate $TH_1(X_R^2)$, since $TH_1(X_R^2)\cong \A(R)/ \langle t^2-1 \rangle$~\cite[Lemma~2.2]{Friedl:2003-4}.
 However, the longitudes of $\eta$ are identified with the meridians of $J$ in $X_{R_{\eta}(J)}$, and so since $\chi$ is a nontrivial (hence surjective) character, the map $\pi_1(X_J) \to \Z_n$ given by $\gamma \mapsto b(\gamma) \in \Z_n$ is surjective. Henceforth, unless otherwise specified, all homology in this proof is taken to be twisted with $\Z[\xi]$-coefficients induced by (restrictions of) the maps $\phi_\chi$ and $\Phi$,  composed with $f_*$ or $g_*$ as appropriate.

We are in the setting of Lemma~\ref{lem:computation1}  and therefore
$H_1(X_J) \cong \A_{\xi}(J)^{1 \oplus\bar{1}}$ and $H_1(X_D) \cong \A_{\xi}(D)^{1 \oplus\bar{1}}$.
The decompositions outlined in Construction~\ref{cons:satellitedisc}
are related by inclusion and degree one maps in such a way that, when we take homology with  twisted $\Z[\xi]$-coefficients, we obtain a  commutative diagram.
Note that the twisted homology $H_1(X_U)= H_1(X_E)= H_1(S^1 \times D^2)=0$, by Lemma~\ref{lem:computation1}, since each of these spaces have trivial Alexander module. Also, the maps $H_0(T^2) \to H_0(X_*)$ for $*=U, J$ and $H_0(S^1 \times D^2) \to H_0(X_*)$ for $*=E,D$ are isomorphisms, as follows from an analysis as in Remark~\ref{rem:h0}.
All horizontal sequences are exact, since they arise from Mayer-Vietoris sequences. We have simplified the following diagram using these observations:
\[\begin{tikzcd}[column sep=1.4em]
 &0 \arrow{dd} \arrow{rr}& &\begin{array}{c} H_1(X_{\Delta_0}) \\ \oplus \\ H_1(X_D) \end{array} \arrow["(\Id  \, 0)" pos=0.8]{dd} \arrow{rr}{(\pi_\Delta \, \pi_D)}& & H_1(X_{\Delta_D}) \arrow["g_*" pos=0.85]{dd} \arrow{r} &0\\
H_1(T^2) \arrow{dd}{\Id} \arrow{ur} \arrow[crossing over, "(j_{\eta} \,  j_J)" pos=.85]{rr}& &\begin{array}{c}  H_1(X_R \ssm \eta) \\ \oplus \\ H_1(X_J) \end{array}  \arrow{ur}{(i_{\eta} \, i_J)} \arrow[crossing over, "(\pi_\eta \, \pi_J)" pos=.7]{rr} && H_1(X_{R_{\eta}(J)}) \arrow{ur}{i}  \arrow[crossing over]{rr}&&0  \\
 &0 \arrow{rr}& & H_1(X_\Delta)  \arrow["\pi" pos=.3]{rr} & & H_1(X_{\Delta}) \arrow{r} &0\\
H_1(T^2) \arrow{ur} \arrow{rr}{j_R=j_{\eta}} & & H_1(X_R \ssm \eta)  \arrow[crossing over, leftarrow, "(\Id \, 0)" pos=0.2]{uu} \arrow{ur}{i_\eta} \arrow{rr}{\pi_R} && H_1(X_{R}) \arrow[crossing over, leftarrow, "f_*" pos=0.15]{uu} \arrow{ur}{i_R} \arrow{r}  & 0.& \\
\end{tikzcd}
\]
For reasons of concision, in the above diagram we use $(f_1 \, \, f_2)$ to variously refer to any of the maps $\left[ \begin{array}{c} f_1 \\ f_2 \end{array} \right]$, $[f_1 \,\, f_2]$, or $\left[ \begin{array}{cc} f_1 & 0 \\0 & f_2 \end{array} \right]$ as appropriate.

We immediately obtain that
\[ [\pi_\Delta \, \, \pi_D] \colon H_1(X_{\Delta_0}) \oplus H_1(X_D) \to H_1(X_{\Delta_D})\] is an isomorphism, which is the second identification of the proposition.
We also see that
\[H_1(X_R)= \Imm(\pi_R) \cong H_1(X_R \ssm \eta)/ \ker(\pi_R)= H_1(X_R \ssm \eta)/ \Imm(j_R)\]
and similarly that
\[H_1(X_{R_\eta(J)})= \Imm([\pi_\eta \, \, \pi_J]) \cong \big(H_1(X_R \ssm \eta)\, \oplus\, H_1(X_J) \big)/ \Imm \left[ \begin{array}{c}j_R \\ j_J\end{array} \right].\]
We can directly compute that
\[H_1(T^2)= H_1(C_*(\widetilde{T^2}) \otimes_{\Z[\pi_1(T^2)]} \Z[\xi]^2)\cong  \left(\Z[\xi]/ (\xi-1)\right)^{1 \oplus\bar{1}}\]
 is generated  as a $\Z[\xi]$-module by $\alpha \otimes [0,1]$ and $\alpha \otimes [1,0]$, where $\alpha$ is the curve on $T^2$ identified with $\mu_{\eta}$ in $X_R \ssm \eta$ and $\lambda_J$ in $X_J$. Since $[\lambda_J]=0 \in H_1(X_J^\infty)$, we see that
\[ j_J( \alpha \otimes [0,1])= j_J(\alpha \otimes [1,0])=0 \text{ in } H_1(X_J)\]
and hence that $j_J=0$.

It follows that the map induced by $[\pi_\eta \, \, \pi_J]$
from  $H_1(X_R \ssm \eta) / \Imm(j_\eta) \, \oplus \, H_1(X_J)$ to $H_1(X_{R_\eta(J)})$
 is an isomorphism, and that our desired isomorphism is given by the composition\footnote{The labels of the maps in Equation~(\ref{eqn:isomorphism}) are mild abuses of notation. In particular, $\pi_R \colon H_1(X_R \ssm \eta) \to H_1(X_R)$ is not itself an isomorphism and hence does not have an inverse until we mod out by $\Imm(j_\eta)$, and $[\pi_\eta \, \, \pi_J]$ actually has domain $H_1(X_R \ssm \eta) \oplus H_1(X_J)$, though it of course induces a well-defined map on $H_1(X_R \ssm \eta) / \Imm(j_\eta) \oplus H_1(X_J)$. Nevertheless, we hope the reader finds the reminder of how these maps are induced sufficiently helpful so as  to outweigh the indignity of slightly misleading labels.}
\begin{align}\label{eqn:isomorphism}
\Phi \colon H_1(X_R) \oplus H_1(X_J) \xrightarrow{\left[\begin{array}{cc} \pi_R^{-1}  & 0 \\ 0 & \Id \end{array} \right]} H_1(X_R \ssm \eta) / \Imm(j_\eta)& \oplus H_1(X_J) \xrightarrow{[\pi_\eta \, \, \pi_J]} H_1(X_{R_\eta(J)}).
\end{align}

It remains to show that $\Phi^{-1}(\ker(i))=\ker(i_R) \oplus \ker(i_J)$, which will follow from some diagram chasing,

\begin{claim}
  $\Phi^{-1}(\ker(i)) \subseteq \ker(i_R) \oplus \ker(i_J)$.
\end{claim}

Let $x \in \ker(i)$. Since $(\pi_\eta \oplus \pi_J)$ is onto, there exists $a \in H_1(X_R \ssm \eta)$ and $b \in H_1(X_J)$ such that $(\pi_\eta \oplus \pi_J)(a,b)=x$. Moreover, $(\pi_R(a), b)= \Phi^{-1}(x)$, so it suffices to show that
\[i_R(\pi_R(a))=0 \in H_1(X_{\Delta_0}) \text{ and }i_J(b)=0 \in H_1(X_D).\]
Observe that by the commutativity of our large diagram,
 \begin{align*}
 \pi_R(a)= (\pi_R \circ [\Id \, \, 0]) (a, b)= (f_*\circ [\pi_\eta \,\, \pi_J])(a,b)= f_*(x).
\end{align*}
Therefore
\begin{align*}
(i_R \circ \pi_R)(a)&=(i_R \circ f_*) (x)= (g_* \circ i)(x)= g_*(0)=0.
\end{align*}
In order to show that $i_J(b)=0$, observe that
\begin{align*}
\left([\pi_\Delta \,\, \pi_D] \circ
\left[ \begin{array}{cc} i_\eta & 0 \\ 0 & i_J \end{array}\right]\right)(a, b)
 = (i \circ [\pi_\eta \,\, \pi_J])(a, b)= i(x)=0.
\end{align*}
But $[\pi_\Delta \,\, \pi_D]$ is an isomorphism, and so it follows that
\[\left[\begin{array}{cc} i_\eta & 0 \\ 0& i_J\end{array} \right](a, b)= (i_\eta(a), i_J(b))=0.\]
So $i_J(b)=0$ as desired. This completes the proof of the claim that $\Phi^{-1}(\ker(i)) \subseteq \ker(i_R) \oplus \ker(i_J)$.
\\

\begin{claim}
  $\Phi^{-1}(\ker(i)) \supseteq \ker(i_R) \oplus \ker(i_J)$.
\end{claim}
It suffices to show that both $\ker(i_R)$ and $\ker(i_J)$ are contained in $\Phi^{-1}(\ker(i))$.
Observe that if $b \in \ker(i_J)$ then
\begin{align*}
i(\Phi(b))= i(\pi_J(b))= \pi_D(i_J(b))= \pi_D(0)=0,
\end{align*}
so $b \in \Phi^{-1}(\ker(i))$.
Now let $\alpha \in \ker(i_R)$ to show that $\Phi(\alpha) \in \ker(i)$. Let $a \in H_1(X_R \ssm \eta)$ be such that $\pi_R(a)= \alpha$, and observe that $\Phi(\alpha)= \pi_{\eta}(a)$.
We have that
\begin{align*}
(\pi \circ i_\eta)(a)=(i_R \circ \pi_R)(a) = i_R(\alpha)=0.
\end{align*}
Since $\pi$ is an isomorphism, this implies that $i_\eta(a)=0$ and hence that
\[i \left(\Phi(\alpha)\right)= i\left(  \pi_\eta(a)\right)= \pi_\Delta \left( i_\eta(a)\right)= \pi_\Delta(0)=0,
\]
as desired.
This completes the proof of the claim that  $\Phi^{-1}(\ker(i)) \supseteq \ker(i_R) \oplus \ker(i_J)$.

The last two claims combine to show that $\Phi^{-1}(\ker(i))=\ker(i_R) \, \oplus\,  \ker(i_J)$, which completes the proof of Proposition~\ref{prop:satellitekernel}.
\end{proof}

Note that given a properly embedded disc $D$ in $D^4$ and a knotted 2-sphere $S$ in $S^4$, we can decompose $X_{D\#S}= X_D \cup_{S^1 \times D^2} X_S$. It follows that the double cover is decomposed analogously; gluing in the branch set and applying a straightforward Mayer-Vietoris argument tells us that
\[H_1(\Sigma_2(D^4, D\#S)) \cong H_1(\Sigma_2(D^4, D)) \oplus H_1(\Sigma_2(S^4, S)).\]
Given  $\chi \colon H_1(\Sigma_2(K))) \to \Z_n$ that extends to $\chi_D\colon H_1(\Sigma_2(D^4, D)) \to \Z_n$, define
\[\chi_{D\#S} \colon H_1(\Sigma_2(D^4, D\#S)) \cong H_1(\Sigma_2(D^4, D)) \oplus H_1(\Sigma_2(S^4, S))  \xrightarrow{\chi_D\, \oplus \, 0} \Z_n.\]
We can now show an analogue of Proposition~\ref{prop:2spheresum} in the context of twisted homology.

\begin{prop}\label{prop:2knotsandtwistedh1}
Let $D$ be a properly embedded disc in $D^4$ with boundary $K$, and let $S$ be a knotted 2-sphere in $S^4$.
Let  $\chi \colon H_1(\Sigma_2(K)) \to \Z_n$ be a map that extends to $\chi_D\colon H_1(\Sigma_2(D^4, D)) \to \Z_n$, and let $\chi_{D\#S}$ be as above.
Then
\[
\ker\left(H_1^{\phi_\chi}(X_K) \to H_1^{\phi_{\chi_D}}(X_D)\right) = \ker\left(H_1^{\phi_\chi}(X_K) \to H_1^{\phi_{\chi_{D\#S}}}(X_{D \#S})\right).
\]
\end{prop}

\begin{proof}
For a submanifold $Y \subset X_{D\#S}$ we can restrict  $\phi_{\chi_{D\#S}}$ to $\pi_1(Y)$ and, by a mild abuse of notation we let $H^{\phi_{\chi_{D\#S}}}_*(Y)$ denote the resulting twisted homology with $\Z[\xi_n]$-coefficients.

We shall use the decomposition $X_{D\#S}= X_D \cup_{S^1 \times D^2} X_S$.
First we compute the homology of $S^1 \times D^2$ and $X_S$.
Letting $t$ denote the generator of $\pi_1(S^2 \times D^2) \cong \Z$, we can pick a cell structure for (a space homotopy equivalent to) $S^1 \times D^2$ consisting of a single 0-cell and a single 1-cell and use this to compute
\begin{align*}
H^{\phi_{\chi_{D \#S}}}_1(S^1 \times D^2)
   &\cong \ker(\phi_{\chi_{D \# S}}(t) - \Id)  \\  &= \ker\left(\begin{bmatrix} 0 & 1 \\ 1 & 0 \end{bmatrix} \begin{bmatrix}\xi^b  &0 \\ 0 & \xi^{-b}  \end{bmatrix} - \begin{bmatrix} 1 & 0 \\ 0 & 1  \end{bmatrix} \colon \Z[\xi_n]^2 \to \Z[\xi_n]^2\right)
\text{, for some } b \in \Z
   \\    & = \ker \left( \begin{bmatrix} -1 & \xi^{-b} \\ \xi^b & -1  \end{bmatrix}\right) \cong \left\{
  (x,y)
  %{\begin{pmatrix} x \\ y \end{pmatrix}}
   \in \Z[\xi_n]^2 \mid \xi^b x =y \right\} \cong \Z[\xi_n].
  \end{align*}
.

%We can also deduce that
\begin{claim}
 We have that
  \[H^{\phi_{\chi_{D \#S}}}_1(X_S) \cong \Z[\xi_n] \oplus \left(\A(S) \otimes_{\Z[t^{\pm1}]} \Z[\xi_n]^2\right), \]
where on the right we have the action of $\Z[t^{\pm1}]$ on $ \Z[\xi_n]^2$ given by $t \cdot [x,\,y]= [y,\, x]$.
\end{claim}

To see this, use the K\"{u}nneth spectral sequence~\cite[Theorem 5.6.4]{Weibel94} as in the proof of Lemma~\ref{lem:computation1}.
Since $H_0(X^{\infty}) \cong \Z$, we obtain
\begin{align*}
E_{0,1}^2&= \A(S) \otimes_{\Z[t^{\pm1}]} \Z[\xi_n]^2\\
E_{1,0}^2&= \Tor_1^{\Z[t^{\pm1}]}(H_0(X_S^{\infty}), \Z[\xi_n]^2) \cong H^{\phi_{\chi_{D \#S}}}_1(S^1) \cong \Z[\xi_n] \\
E_{2,0}^2&= \Tor_2^{\Z[t^{\pm1}]}(H_0(X_S^{\infty}), \Z[\xi_n]^2) \cong H^{\phi_{\chi_{D \#S}}}_2(S^1) = 0.
\end{align*}
Since $E^2_{2,0} =0$ it follows that $E^2_{0,1} \cong E^{\infty}_{0,1}$. We also have $E^2_{1,0} \cong E^\infty_{1,0}$.
The spectral sequence therefore gives rise to a short exact sequence of $\Z[\xi_n]$-modules
\[0 \to \A(S) \otimes_{\Z[t^{\pm1}]} \Z[\xi_n]^2 \to H^{\phi_{\chi_{D \#S}}}_1(X_S) \to \Z[\xi_n] \to 0,\]
which splits since the last module is free.
This completes the proof of the claim.

Moreover, comparing the spectral sequences for $S^1 \times D^2$ and $X_S$ using naturality, it follows that the map $\Z[\xi_n] \cong H^{\phi_{\chi_{D \# S}}}_1(S^1 \times D^2) \to H^{\phi_{\chi_{D \# S}}}_1(X_S)$ is injective and maps onto $\Z[\xi_n]$.

Since the restriction of \[\phi_{\chi_{D\#S}} \colon \pi_1(X_{D \# S}) \to \Z_2 \ltimes \Z_n\] to $\pi_1(X_S)$ is the map $\gamma \mapsto ([\varepsilon_S(\gamma)], 0)$ we have that
\[H_0^{\phi_{\chi_{D \#S}}}(S^1 \times D^2) \to H_0^{\phi_{\chi_{D \#S}}}(X_S)\]
 is an isomorphism, see Remark~\ref{rem:h0}.
The Mayer-Vietoris sequence for $X_{D\#S}= X_D \cup_{S^1 \times D^2} X_S$ with $\Z[\xi_n]$-coefficients therefore gives us that
 \[H^{\phi_{\chi_{D\# S}}}_1(X_{D \#S}) \cong H^{\phi_{\chi_D}}_1(X_D) \oplus \left(\A(S) \otimes_{\Z[t^{\pm1}]} \Z[\xi_n]^2\right),\]
 since $H^{\phi_{\chi_{D \#S}}}_1(S^1 \times D^2) \cong \Z[\xi_n]$ maps onto the $\Z[\xi_n]$-summand of $H_1^{\phi_{\chi_{D \#S}}}(X_S)$.

  Since $X_K \subset X_D$, the inclusion induced map $H^{\phi_{\chi}}_1(X_K) \to H^{\phi_{\chi_{D \# S}}}_1(X_{D\# S})$ factors as
  \[H^{\phi_{\chi}}_1(X_K) \to H^{\phi_{\chi_{D}}}_1(X_{D}) \to H^{\phi_{\chi_{D}}}_1(X_{D}) \oplus (\A(S) \otimes_{\Z[t^{\pm1}]} \Z[\xi_n]^2) \cong  H^{\phi_{\chi_{D \# S}}}_1(X_{D\# S}).\]
  We saw that the central map is a split injection, the inclusion of the $H^{\phi_{\chi_{D}}}_1(X_{D})$ direct summand.  It follows that
 \[ \ker(H^{\phi_{\chi}}_1(X_K) \to H^{\phi_{\chi_{D\# S}}}_1(X_{D \#S})) \cong \ker(H^{\phi_{\chi}}_1(X_K) \to H^{\phi_{\chi_D}}_1(X_{D}))\]
 as desired.
\end{proof}

\subsection{Construction of examples and proof of Theorem~\ref{thm:C}}

 Recall from Notation~\ref{notation:Axi} that for a space $X$ and a root of unity $\xi$, we define
\[A_{\xi}(X):=\mathcal{A}(X) \otimes_{\Z[t^{\pm1}]} \Z[\xi].\]
Now let $J_0$ be a ribbon knot with preferred ribbon disc $D_0$
 such that
\[\A_{\xi_3}(J_0)/ \ker\left( \A_{\xi_3}(J_0) \to  \A_{\xi_3}(D_0)\right)\]
is nonzero.
The knot $J:= J_0 \# -J_0$ has two preferred slice (in fact ribbon) discs: $D_1$ consists of $D_0 \natural -D_0$ and $D_2$ is the standard ribbon disc for any knot of the form $K \# -K$ obtained by spinning.
Note that $ \A(J) \cong\A(J_0) \oplus \A(J_0)$, $\A(D_1)\cong \A(D_0) \oplus \A(D_0)$, and by the next lemma
$\A(D_2)\cong \A(J_0)$.

\begin{lem}\label{lem:spun}
  The spun slice disc satisfies $\A(D_2)\cong \A(J_0)$.
\end{lem}

\begin{proof}
 Let $J_0^{\dag}$ be a tangle $D^1 \subseteq D^3$ arising from removing a trivial ball-arc pair $(D^3,D^1)$ from $(S^3,J_0)$. Note that \[\A(J_0^{\dag}) = H_1(D^3\setminus \nu J_0^{\dag}) \cong \A(J_0)\] and \[D^4 \setminus \nu D_2 \cong (D^3 \setminus \nu J_0^{\dag}) \times I \simeq D^3 \setminus \nu J_0^{\dag}.\]
 It follows that $\A(D_2)\cong \A(J_0)$ as claimed.
\end{proof}

Moreover, the map
$i_1 \colon \A(J)\to \A(D_1)$ is given by $(x,y) \mapsto (i_0(x), i_0(y))$ and
the map $i_2 \colon \A(J) \to \A(D_2)$ is given by $(x,y) \mapsto x+y$.

\begin{exl}\label{exl:61again}
One example of such a knot is $J_0= 6_1$.  As noted in Example~\ref{exl:61},  $\A(J_0)=  \Z[t^{\pm1}]/ \langle (2t-1)(t-2) \rangle$, $\A(D_0)=\Z[t^{\pm1}]/ \langle t-2 \rangle$ and the map $i_0 \colon \A(J_0) \to \A(D_0)$  is given by multiplication by $2t-1$.
In particular, we have that
\begin{align*}
\A_{\xi_3}(J_0)/ \ker\left( \A_{\xi_3}(J_0) \to  \A_{\xi_3}(D_0)\right) &\cong \Z[\xi_3]/ \langle (2 \xi_3-1)(\xi_3-2), \xi_3-2 \rangle\\
& \cong \Z_7[x]/ \langle x-2 \rangle \neq 0.
\end{align*}
Here the $\Z_7$ comes from $\xi_3^2 + \xi_3+1=0$, combined with $\xi_3-2=0$.
\end{exl}

Now we prove the following more explicit version of Theorem~\ref{thm:C}.

\begin{thm}\label{thm:subtlerlargedistancedetail}
Let $(R, \eta, \Delta_0)$ be as in Example~\ref{exl:61} and let  $J_0$ be a ribbon knot with  preferred ribbon disc $D_0$ such that
$\A_{\xi_3}(J_0)/ \ker\left( \A_{\xi_3}(J_0) \to  \A_{\xi_3}(D_0)\right)$ is nonzero.  Let $J=J_0 \#-J_0$, $D_1$, and $D_2$ be defined as above.
Then for any $g\geq 0$,
the knot $\displaystyle K:= \#_{i=1}^{4g} R_{\eta}(J)$ has ribbon discs $\Delta_1$, the boundary connected sum of $4g$ copies of $\Delta_{D_{1}}$, and $\Delta_2$, the boundary connected sum of $4g$ copies of $\Delta_{D_2}$,  such that
\[ \ker(\A_{\Q}(K) \to \A_{\Q}(\Delta_1)) \cong \ker(\A_{\Q}(K) \to \A_{\Q}(\Delta_2)).\]
and yet
\[d_2(\Delta_1, \Delta_2) \geq g.\]
\end{thm}

As discussed in Example~\ref{exl:61}, since both $\Delta_{D_1}$ and $\Delta_{D_2}$ are obtained from surgery on a genus 1 Seifert surface for $R_{\eta}(J)$, we know that $d_2(\Delta_{D_1}, \Delta_{D_2}) \leq 1$.
It follows that $d_2(\Delta_1, \Delta_2) \leq 4g$, though we are not able to determine $d_2(\Delta_1, \Delta_2)$ precisely.
%\footnote{\allison{ I deleted `due to complications extending representations.' It's even more than that- we also have an extra factor of 2 coming from us taking degree 2 representations.}}

\begin{remark}
The proof that  $d_2(\Delta_1, \Delta_2) \geq g$ is somewhat long and involved, so for the reader's convenience we outline the key points in advance:

We suppose that $F$ is a genus $h \leq g$ surface to which both $\Delta_1$ and $\Delta_2$ stabilize by addition of $h$ 1-handles and some number of local 2-knots, in order to show $h=g$.

For $j=1,2$ let $\Delta_j'$ be a disc obtained from $\Delta_j$ by 2-knot addition which stabilizes to $F$ via $h$ 1-handle additions.
 Let $T= T_1 \cup -T_2$ denote the standard cobordism built as in Construction~\ref{cons:standardcobordism}, so $X_T$ is a cobordism from $X_{\Delta_1'}$ through $X_F$  to $X_{\Delta_2'}$.
Our first main argument proving Claim~\ref{claim:1} below shows that there exists a highly nontrivial character on $H_1(\Sigma_2(K))$ giving rise to a representation $\pi_1(X_K) \to \Z_2 \ltimes \Z_3$ that extends over $X_T$ to a map $\Phi$ with certain nice properties.

Just as in the proof of Theorem~\ref{thm:B}, we compare $\ker(H_1^{\Phi}(X_K) \to H_1^{\Phi}(X_{\Delta_1}))$ and $\ker(H_1^{\Phi}(X_K) \to H_1^{\Phi}(X_{\Delta_2}))$.
Essentially by Proposition~\ref{prop:2knotsandtwistedh1} and the careful construction of~$\Phi$, we are able to work with $\ker(H_1^{\Phi}(X_K) \xrightarrow{\iota_1} H_1^{\Phi}(X_{\Delta_1'}))$ and $\ker(H_1^{\Phi}(X_K) \xrightarrow{\iota_2}  H_1^{\Phi}(X_{\Delta_2'}))$ instead.
By the construction of our examples, work before the statement of Theorem~\ref{thm:subtlerlargedistancedetail}, and Proposition~\ref{prop:satellitekernel},  we can show that  $\ker(\iota_2)/ (\ker(\iota_1) \cap \ker(\iota_2))$ has generating rank $x$ at least $2g$. We then  use Proposition~\ref{prop:kerneloftwistedH1} to show that $\ker(\iota_F)$ both contains $\ker(\iota_2)$ and is generated by $\ker(\iota_1)$ together with some other $2h$ elements. It follows that $\ker(\iota_2)/ (\ker(\iota_1) \cap \ker(\iota_2))$  has generating rank $x$ no more than $2h$, and hence $2g \leq x \leq 2h$ so $g \leq h$. We assumed $h \leq g$ so $g=h$ as desired.
\end{remark}

\begin{proof}[Proof of Theorem~\ref{thm:subtlerlargedistancedetail}]
Fix $g \in \mathbb{N}$, and let $K$, $\Delta_1$, and $\Delta_2$ be as above.
Define $N=4g$,  $\xi:=\xi_3$, and recall that for any knot or slice disc $L$ we have $\A_{\xi}(L):=\A(L) \otimes_{\Z[t^{\pm1}]} \Z[\xi]$.
By Proposition~\ref{prop:homologycoverinfection} we have identifications
 \begin{align*}
 \A(K) &\cong \bigoplus_{i=1}^N \A(R_{\eta}(J)) \cong \bigoplus_{i=1}^N \A(R)\\
\text{ and } \A(\Delta_j)& \cong \bigoplus_{i=1}^N \A(\Delta_{D_j}) \cong  \bigoplus_{i=1}^N \A(\Delta_0) \text{ for }j=1,2
 \end{align*}
in such a way that $\ker(\A(K) \to \A(\Delta_1))$ and $\ker(\A(K) \to \A(\Delta_2))$ are both identified with a sum $\bigoplus_{i=1}^N \ker( \A(R) \to \A(\Delta_0))$, and in particular are equal. Since $\A_{\Q}(L) \cong \A(L) \otimes \Q$ for any knot or slice disc $L$, our first conclusion follows.\\

Now suppose that $F$ is a genus $h \leq g$ surface to which both $\Delta_1$ and $\Delta_2$ stabilize by addition of $h$ 1-handles and some number of local 2-knots.
We shall show under these assumptions that $h \geq g$.
As in the proof of Proposition~\ref{prop:showing-d_2-large}, for $j=1,2$ there exist discs $\Delta_j'$ obtained from $\Delta_j$ by connected sum with  local 2-knots such that $F$ is obtained from $\Delta_j'$ by $h$ 1-handle additions. For $j=1,2$ we write $\Delta_j'= \Delta_j \# S_j$ for some local 2-knot $S_j$.

Note that $f \colon X_{R_{\eta}(J)} \to X_R$ lifts to give a degree one map $X_{R_{\eta}(J)}^2 \to X_R^2$, which extends to give
$\widebar{f}\colon \Sigma_2(R_{\eta}(J)) \to \Sigma_2(R)$.
Moreover, Proposition~\ref{prop:homologycoverinfection} implies that $\widebar{f}$ induces an isomorphism on first homology.
So we obtain an isomorphism
\[\mathfrak{f} \colon H_1(\Sigma_2(K)) \cong \oplus_{i=1}^N H_1(\Sigma_2(R_{\eta}(J)))\xrightarrow{\oplus_{i=1}^N \widebar{f}_*} \oplus_{i=1}^N H_1(\Sigma_2(R)) \cong H_1(\Sigma_2(R_N))
\]
where we let $R_N$ denote the connected sum of $N$ copies of $R$.

Let $T_1$ and $T_2$ be appropriate unions of the simple cobordisms built
 in Construction~\ref{cons:standardcobordism}, such that $X_{T_1}$ is a cobordism from $X_{\Delta_1'}$ to $X_F$ rel.\ $X_K$ and $X_{T_2}$ is a cobordism from $X_{\Delta_2'}$ to $X_F$ rel.\ $X_K$. We let  $X_T:=X_{T_1} \cup_{X_F} - X_{T_2}$.

\begin{claim}\label{claim:1}
There exists a map
\[\chi =(\chi_i)_{i=1}^N \colon \oplus_{i=1}^N H_1(\Sigma_2(R)) \to \Z_3\]
 with at least $2g$ of the $\chi_i$ nonzero  such that
 $\phi_{\chi \circ \mathfrak{f}} \colon \pi_1(X_K) \to \Z_2 \ltimes \Z_3$ extends over $\pi_1(X_T)$ to a map $\Phi \colon \pi_1(X_T) \to \Z_2 \ltimes \Z_3$ and for $j=1,2$ the composition
\[\pi_1(X_{S_j}) \to \pi_1(X_{\Delta_j}) *_{\Z} \pi_1(X_{S_j}) \cong \pi_1(X_{\Delta_j'}) \to \pi_1(X_T) \xrightarrow{\Phi} \Z_2 \ltimes \Z_3
\]
is given by $\gamma \mapsto ([\varepsilon(\gamma)], 0)$.
  \end{claim}

We will always construct our extensions in stages, first extending over
\[Y= X_{\Delta_1'} \cup (X_K \times I) \cup X_{\Delta_2'}\]
and then extending over the rest of $X_T$.

Note that $H_1(\Sigma_2(R)) \cong \Z_9$ and that it follows from Proposition~\ref{prop:homologycoverinfection} that
\begin{align}\label{eqn:branchedcoverhomology}
\ker\big(H_1(\Sigma_2(K)) \to H_1(\Sigma_2(D^4, \Delta_j))\big)& \cong \ker\Big(\bigoplus_{i=1}^N H_1(\Sigma_2(R)) \to \bigoplus_{i=1}^N H_1(\Sigma_2(D^4, \Delta_0))\Big) \\ &\cong\bigoplus_{i=1}^N 3 \Z_9.
\end{align}

It follows that for $j=1,2$ and for any character $\chi\colon H_1(\Sigma_2(R_N)) \to \Z_3$ we have that $\chi \circ \mathfrak{f}$ extends to a map $\chi_j$ on $H_1(\Sigma_2(D^4, \Delta_j))$, up to a priori extending its range to $\Z_{3^a}$ for some $a\geq 1$. However, since our slice discs $\Delta_j$ are in fact ribbon discs, the inclusion induced map $\pi_1(X_K) \to \pi_1(X_{\Delta_j})$ is surjective for $j=1,2$. So we can take $a=1$.
%
% Moreover, we have by a Mayer-Vietoris argument that for $j=1,2$
% \[H_1(\Sigma_2(D^4, \Delta_j')) \cong  H_1(\Sigma_2(D^4, \Delta_j)) \oplus H_1(\Sigma_2(S^4, S_j))
% \]
%So for any $\chi\colon H_1(\Sigma_2(K)) \to \Z_3$ we have a prescribed way to extend $\chi$ to a map
%\[\chi_j' \colon H_1(\Sigma_2(D^4, \Delta_j')) \xrightarrow{\pi} H_1(\Sigma_2(D^4, \Delta_j')) \xrightarrow{\chi_j} \Z_3,
%\] by requiring $\chi_j'$ to be trivial on the $H_1(\Sigma_2(S^4, S_j))$-summand.
%

Note that any map $\chi \circ \mathfrak{f}  \colon H_1(\Sigma_2(K)) \to \Z_3$ induces
$\widebar{\chi \circ \mathfrak{f}} \colon H_1(X_K^2) \to \Z_3$ by precomposition with the natural inclusion induced map $H_1(X_K^2) \to H_1(\Sigma_2(K))$.
Since inclusion induces isomorphisms of $H_1(X_K)$ with $H_1(X_T)$, in order to show that a given $\phi_{\chi \circ \mathfrak{f}}$ extends over $\pi_1(X_T)$ it suffices to extend the corresponding $\widebar{\chi \circ \mathfrak{f}}$
first over $\pi_1(X_{\Delta_1'}^2 \cup (X_K^2 \times I) \cup X_{\Delta_2'}^2)$ and then
over $\pi_1(X_T^2)$.
%This occurs  if and only if $\widebar{\chi}$ vanishes on $\ker(H_1(X_K^2) \to H_1(X_T^2))$.
%(Since $\pi_1(X_{\Delta_j'}) \to \pi_1(X_{T_j})$ is surjective, there is no need to expand our codomain at this latter stage.)

Now, consider the Mayer-Vietoris sequence for $X_{\Delta_1'}^2 \cup (X_K^2 \times I) \cup X_{\Delta_2'}^2$, which we note is diffeomorphic to $X_{\Delta_1'}^2 \cup_{X_K^2} X_{\Delta_2'}^2$:
\[
H_1(X_K^2) \xrightarrow{i_1' \oplus i_2'} H_1(X_{\Delta_1'}^2 ) \oplus H_1(X_{\Delta_2'}^2) \xrightarrow{j_1 \oplus j_2} H_1( X_{\Delta_1'}^2 \cup_{X_K^2} X_{\Delta_2'}^2) \to 0.
\]

For $j=1,2$ we have that  $H_1(X_{\Delta_j'}^2) \cong H_1(X_{\Delta_j}^2) \oplus H_1(\Sigma_2(S^4, S_j))$ in such a way that $i_j' \colon H_1(X_K^2) \to H_1(X_{\Delta_j'}^2)$ is given by $i_j \,\oplus\, 0$, where $i_j \colon H_1(X_K^2) \to H_1(X_{\Delta_j}^2)$ is the inclusion-induced map.
We therefore obtain, recalling that  the map $H_1(X_K^2) \to H_1(X_{\Delta_j}^2)$ is surjective since $\Delta_j$ is a ribbon disc, that
\[ H_1( X_{\Delta_1'}^2 \cup_{X_K^2} X_{\Delta_2'}^2) \cong
H_1(X_{\Delta_1}^2) \oplus H_1(\Sigma_2(S^4, S_1)) \oplus H_1(\Sigma_2(S^4, S_2)).
\]

Therefore any $\widebar{\chi \circ \mathfrak{f}}$  can be extended over
\begin{align*}
X_{\Delta_1'}^2 \cup (X_K^2 \times I) \cup X_{\Delta_2'}^2
&=(X_{\Delta_1}^2 \cup X_{S_1}^2) \cup (X_K^2 \times I) \cup (X_{\Delta_2}^2 \cup X_{S_2}^2) \subset \partial X_{T}^2
\end{align*}
so that the extension is trivial on the $H_1(\Sigma_2(S^4, S_1)) \oplus H_1(\Sigma_2(S^4, S_2))$-summand.
Moreover, such a map extends over $H_1(X_T^2)$ if and only if it vanishes on
\[H:= \ker \big(
H_1(X_{\Delta_1'}^2 \cup (X_K^2 \times I) \cup X_{\Delta_2'}^2)
\to H_1(X_T^2)
\big).
\]
Note that our maps $\widebar{\chi \circ \mathfrak{f}}$ have been chosen to vanish on $H_1(\Sigma_2(S^4, S_1)) \oplus H_1(\Sigma_2(S^4, S_2))$, and hence vanish on $H$ if and only if they vanish on
\[H \cap H_1(X_{\Delta_1}^2)= \ker \left( H_1(X_{\Delta_1}^2) \to H_1(X_T^2) \right).\]
Moreover, $\ker \left( H_1(X_{\Delta_1}^2) \to H_1(X_T^2) \right)$ is isomorphic to a quotient of $\ker(H_1(X_K^2) \to H_1(X_T^2))$.

For a space $X$ with surjection $\varepsilon \colon H_1(X) \to \Z$, we consider the map
\begin{align*}
  e=e_X \colon \pi_1(X) &\to \GL_2(\Z) \\ \gamma &\mapsto \left[ \begin{array}{cc}0& 1 \\ 1 & 0 \end{array}\right]^{\varepsilon(\gamma)}.
\end{align*}
Note that the $e_X$ maps for $X= X_K, X_{\Delta_j'}, X_{F}, X_T$ are compatible, since inclusion $X_K \hookrightarrow X_*$ induces an isomorphism on first homology.
The proof of Proposition~\ref{prop:kerneloftwistedH1} implies that
\[ \ker(H_1^e(X_K) \to H_1^e(X_{T_1})) \cong
\ker(H_1^e(X_K) \to H_1^e(X_F)) \cong \ker(H_1^e(X_K) \to H_1^e( X_{T_2})).
\]
Proposition~\ref{prop:kerneloftwistedH1} also tells us that this kernel is generated by $\ker(H_1^e(X_K) \to H_1^e(X_{\Delta_1'}))$ along with some $2h$ elements $\{x_k\}_{k=1}^{2h} \subseteq H_1^e(X_K)$.

By the topologists' Shapiro lemma~\cite[p.~100]{Davis-Kirk}, there is a canonical identification $H_1^e(X) \cong H_1(X^2)$ for all $X$, and so
\[\ker(H_1(X_K^2) \to H_1(X_{T_1}^2)) \cong \ker(H_1(X_K^2) \to H_1(X_F^2)) \cong \ker(H_1(X_K^2) \to H_1( X_{T_2}^2)) \]
and this kernel is generated by $\ker(H_1(X_K^2) \to H_1(X_{\Delta_1'}^2))$ along with some $2h$ elements $\{x_k\}_{k=1}^{2h} \subseteq H_1(X_K^2)$.
%the previous sentence remains true when $H_1^e(X_*)$ is replaced with $H_1(X_*^2)$ for $* \in \{K, \Delta_1', \Delta_2', T_1, T_2, F\}$.

Therefore, since every map $H_1(X_K^2) \to \Z_3$ extends over $H_1(X_{\Delta_1'}^2 \cup_{X_K^2} X_{\Delta_2'}^2)$ in our prescribed fashion, in order to ensure that $\widebar{\chi \circ \mathfrak{f}}$ extends over $H_1(X_T^2)$
it is enough to have $(\widebar{\chi\circ \mathfrak{f}})(x_k)=0$ for all $k=1, \dots, 2h$.
It follows from Equation~(\ref{eqn:branchedcoverhomology}) that $\Hom(H_1(\Sigma_2(K)), \Z_3) \cong \Z_3^N$. Using our assumption that $h \leq g$, we have
\[N-2h=(4g)-2h \geq (4g)-2g \geq 2g.\]

A linear algebraic argument as in the proof of~\cite[Theorem~6.1]{KimLivingston:2005} shows that if $A$ is an abelian group with $\Hom(A, \mathbb{F}) \cong \mathbb{F}^N$ then, given any $m$ elements $a_1, \dots, a_m \in A$ there exists a character $\chi= (\chi_i)_{i=1}^N\in \Hom(A, \mathbb{F})$ such that $\chi(a_j)=0$ for all $j=1, \dots, m$ and such that at least $N-m$ of the $\chi_i$ maps are nonzero.
It therefore follows that there exists some $\chi= (\chi_i)_{i=1}^N$ such that $\chi \circ \mathfrak{f}$ vanishes on $\{x_1, \dots, x_{2h}\}$ and at least $N-2h \geq 2g$ of the $\chi_i$ are nonzero. This completes the proof of Claim~\ref{claim:1}.
\\

Let $\chi=(\chi_i)_{i=1}^N$ be such a map. By reordering the summands, without loss of generality we may assume that $\chi_1, \dots, \chi_m$ are nonzero for some $m \geq 2g$ and that $\chi_{m+1}, \dots, \chi_N$ are zero.
Let $\phi:=\phi_{\chi \circ \mathfrak{f}}$
and let $\Phi \colon \pi_1(X_T) \to \Z_2 \ltimes \Z_{3}$ be the corresponding extension of $\phi$ over $\pi_1(X_T)$.

Observe that $X_K$ is the union of $N$ copies of $X_{R_\eta(J)}$, glued along $(N-1)$ copies of~$S^1 \times I$, and that, for $j=1,2$, $X_{\Delta_j'}$ is the union of $N$ copies of $X_{\Delta_{D_j}}$, glued along $(N-1)$ copies of $S^1 \times I \times I$, along with a single copy of $X_{S_j}$ glued along $S^1 \times D^2$ away from all the other identifications.  These decompositions are compatible.

Let $\phi_i$ denote the restriction of $\phi$ to the fundamental group of the $i$th copy of $X_{R_{\eta}(J)}$ and respectively let $\Phi_i$ denote the restriction of $\Phi$ to the $i$th copy of $\pi_1(X_{\Delta_{D_j}})$. Recall that there are some choices of basepoints and paths implicit here -- see the note at the end of Construction~\ref{cons:standardcobordism}.
 It is then straightforward to argue that our maps are related by the following commutative diagram, where unlabeled arrows are induced by inclusion and $\Phi_{\chi_i}$ denotes the unique extension of $\phi_{\chi_i}$ to $\pi_1(X_{\Delta_{0}})$:
\[
\begin{tikzcd}
\pi_1(X_{R_{\eta}(J)}) \arrow{d} \arrow[swap, "f_*"]{r} \arrow[bend left=20, "\phi_i"]{rr} & \pi_1(X_R) \arrow{d} \arrow[swap, "\phi_{\chi_i}"]{r} & \Z_2 \ltimes \Z_3 \arrow["="]{d} \\
\pi_1(X_{\Delta_{D_j}}) \arrow[bend right=20, swap, "\Phi_i"]{rr}\arrow["g_*"]{r} & \pi_1(X_{\Delta_0}) \arrow["\Phi_{\chi_i}"]{r}& \Z_2 \ltimes \Z_3.
\end{tikzcd}
\]

%%%%%%%%%%%%%%%%%%%%%%%%%%%%
%%%%%%%%%%%%%%%%%%%%%%%%%%%%

For $1 \leq i \leq m$, the map $\chi_i$ is nontrivial and so Proposition~\ref{prop:satellitekernel} implies that
\[H_1^{\phi_i}(X_{R_\eta(J)}) \cong H_1^{\phi_{\chi_i}}(X_{R})\oplus \A_{\xi}(J)^{1 \oplus \bar{1}} \text{
and }
H_1^{\Phi_i}(X_{\Delta_{D_j}}) \cong H_1^{\Phi_{\chi_i}}(X_{\Delta_0})\oplus \A_{\xi}(D_j)^{1 \oplus \bar{1}} \]
in such a way that
$\ker(H_1^{\phi_i}(X_{R_\eta(J)})  \to H_1^{\Phi_i}(X_{\Delta_{D_j}}))$ is identified with
\[\ker\left(H_1^{\phi_{\chi_i}}(X_{R}) \to  H_1^{\Phi_{\chi_i}}(X_{\Delta_0})\right) \oplus \ker (\A_{\xi}(J) \to \A_{\xi}(D_j))^{1 \oplus \bar{1}}.
\]

Now consider a portion of the Mayer-Vietoris sequences in twisted homology for $X_{K}= \cup_{i=1}^N X_{R_{\eta}(J)}$ and
$X_{\Delta_j}= \cup_{i=1}^N X_{\Delta_{D_j}}$ for $j=1, 2$:
\[
\begin{tikzcd}
\oplus_{i=1}^{N-1} H_1^{\phi_{i}}(S^1 \times I)  \arrow{d}{\Id} \arrow{r}{u} &\oplus_{i=1}^N H_1^{\phi_{i}}(X_{R_{\eta}(J)})  \arrow{d}{\oplus_{i=1}^n \iota^i_j} \arrow{r}{v} &H_1^{\phi}(X_K) \arrow{d}{\iota_j}  \\
\oplus_{i=1}^{N-1}H_1^{\Phi_i} (S^1 \times I \times I)   \arrow{r}{U_j} &\oplus_{i=1}^N H_1^{\Phi_i}(X_{\Delta_{D_j}})
  \arrow{r}{V_j} &H_1^{\Phi}(X_{\Delta_j}).
\end{tikzcd}
\]
In the above diagram, by a mild abuse of notation we refer to the restriction of $\phi_i$ to $\pi_1(S^1 \times I)$  as just $\phi_i$, and similarly for $\Phi_i|_{\pi_1(S^1 \times I \times I)}$.
\\

We wish to show that $\ker(\iota_2)/ (\ker(\iota_1) \cap \ker(\iota_2))$ has generating rank at least $2g$. In order to do this, we focus on a submodule $Q$ of $\oplus_{i=1}^N H_1^{\phi_{i}}(X_{R_{\eta}(J)})$ and analyze how $v(Q)$ intersects $\ker(\iota_1)$ and $\ker(\iota_2)$.

\begin{claim}\label{claim:2}
The module $Q:= \oplus_{i=1}^m \A_{\xi}(J)^{1 \oplus \bar{1}}\subset \oplus_{i=1}^N H_1^{\phi_{i}}(X_{R_{\eta}(J)})$ is carried isomorphically by $v$  to a subgroup of  $H_1^{\phi}(X_K)$ such that for $q \in Q$ we have that $v(q) \in \ker(\iota_j)$ if and only if $q \in \ker\big(\oplus_{i=1}^N \iota^i_j \big)$.
\end{claim}

First, use Proposition~\ref{prop:kerneloftwistedH1} to  decompose
\[\bigoplus_{i=1}^N H_1^{\phi_{i}}(X_{R_{\eta}(J)}) \cong \bigoplus_{i=1}^m \big( H_1^{\phi_{\chi_i}}(X_R) \oplus  \A_{\xi}(J)^{1 \oplus \bar{1}}\big) \oplus \bigoplus_{i=m+1}^N H_1^{\phi_{i}}(X_{R_{\eta}(J)}).
\]
We can then observe that since
\[(S^1 \times I)_i \subset (X_R)_i \cap (X_R)_{i+1} \subset (X_{R_{\eta}(J)})_i \cap  (X_{R_{\eta}(J)})_{i+1}\]
we have
\[ \ker(v)=\Imm(u) \subseteq \bigoplus_{i=1}^m  H_1^{\phi_{\chi_i}}(X_R) \oplus \bigoplus_{i=m+1}^N H_1^{\phi_{i}}(X_{R_{\eta}(J)}).\]
Similarly, we have that
\[ \ker(V_j)= \Imm(U_j) \subseteq \bigoplus_{i=1}^m  H_1^{\Phi_{\chi_i}}(X_{\Delta_0}) \oplus \bigoplus_{i=m+1}^N H_1^{\Phi_i}(X_{\Delta_{D_j}}).
\]
That is, $\ker(v)$ and $\ker(V_j)$ respectively intersect the $\A_{\xi}(J)^{1 \oplus \bar{1}}$ and $\A_{\xi}(D_j)^{1 \oplus \bar{1}}$ summands trivially.

In order to show that $\iota^i_j(x)=0$ if and only if $\iota_j(v(x))=0$, suppose that $x$ is an element of the $i$th copy of $ \A_{\xi}(J)^{1 \oplus \bar{1}}$ for some $1 \leq i \leq m$.
One direction follows immediately from the commutativity of our diagram: if $\iota^i_j(x)=0$, then
$\iota_j(v(x))= V_j(\iota^i_j(x))= V_j(0)=0.$
So suppose now that $\iota_j(v(x))=0$. It follows that $\iota^i_j(x) \in \ker(V_j)= \Imm(U_j)$, and so there exists $y \in \oplus_{i=1}^{n-1} H_1(S^1)$ such that $U_j(y)= \iota^i_j(x)$. Observe that $\iota^i_j(x-u(y))= \iota^i_j(x)-U_j(y)=0$, so $x-u(y) \in \ker(\iota^i_j)$.
However, since
\[ \iota^i_j(x)\in \bigoplus_{i=1}^m \A_{\xi}(D_j)^{1 \oplus \bar{1}}\] and
\[\iota^i_j(u(y))= U_j(y) \in  \Imm(U_j) \subseteq \bigoplus_{i=1}^m  H_1^{\Phi_{\chi_i}}(X_{\Delta_0}) \oplus \bigoplus_{i=m+1}^N H_1^{\Phi_i}(X_{\Delta_{D_j}})\]
 we must have $\iota^i_j(x)=0=U_j(y)$, as desired.  This completes the proof of Claim~\ref{claim:2}.
\vspace{.5cm}

For $j=1,2$ we have by Claim~\ref{claim:2} that
\begin{align}\label{eqn:pjsplit}
P_j:=v(Q) \cap \ker(\iota_j) \cong Q \cap v^{-1}(\ker(\iota_j)) =Q \cap \bigoplus_{i=1}^m   \ker(\iota_j^i).
\end{align}
We  now argue  that the subset $P_2/ (\ker(\iota_1) \cap P_2)$ of $\ker(\iota_2)/ (\ker(\iota_1) \cap \ker(\iota_2))$
has generating rank at least $2g$,  noting that  by Lemma~\ref{lemma:gen-rank-facts}~(\ref{item:gen-rank-subgroup}) this implies as desired that $\ker(\iota_2)/ (\ker(\iota_1) \cap \ker(\iota_2))$ has generating rank at least $2g$.

By the splitting of the kernel from Proposition~\ref{prop:satellitekernel} we have that
\begin{align}\label{eqn:pjsplit2}
Q \cap \bigoplus_{i=1}^m   \ker(\iota_j^i)=  \bigoplus_{i=1}^m \A_{\xi}(J)^{1 \oplus \bar{1}}\, \cap \, \bigoplus_{i=1}^m   \ker(\iota_j^i)
= \bigoplus_{i=1}^m  \ker \left(
\A_{\xi}(J)^{1 \oplus \bar{1}} \to \A_{\xi}(D_j)^{1 \oplus \bar{1}} \right).
\end{align}
From our computations of the maps $\A_{\xi}(J) \to \A_{\xi}(D_j)$ {before the statement of Theorem~\ref{thm:subtlerlargedistancedetail}}, we also have
\begin{align}\label{eqn:pjcomp}
 \ker\left(\A_{\xi}(J)^{1 \oplus \bar{1}}  \to \A_{\xi}(D_j)^{1 \oplus \bar{1}} \right) =
 \begin{cases}
   \ker(\iota_0^{\xi} \colon \A_{\xi}(J_0) \to \A_{\xi}(D_0))^{1 \oplus \bar{1}}&  j=1\\
\{ (x,-x) \mid x \in \A_{\xi}(J_0)\} &  j=2.
\end{cases}
\end{align}
%We will now show that the generating rank of   $P_2/ \left( \ker(\iota_1) \cap P_2\right)$ is at least $2g$, thereby completing our proof.
Observe that by Claim~\ref{claim:2}  together with Equations~(\ref{eqn:pjsplit}) and~(\ref{eqn:pjcomp}) we have
 \begin{align*}
P_2/ \left( \ker(\iota_1) \cap P_2 \right)&= P_2 / \left(\ker(\iota_1) \cap v(Q) \cap \ker(\iota_2) \right) \\
&=  P_2/ \left( P_2 \cap P_1 \right)\\
& \cong \bigoplus_{i=1}^m \{ (x,-x) \mid x \in \A_{\xi}(J_0)\}/ \bigoplus_{i=1}^m \big\{ (x,-x)\mid x \in \ker(\iota_0^{\xi}) \big\} \\
&\cong \bigoplus_{i=1}^m \A_{\xi}(J_0)/ \ker(\iota_0^{\xi}).
\end{align*}
Since  $\A_{\xi}(J_0)/ \ker(\iota_0^{\xi})$ is nonzero, the classification theorem of finitely generated modules over commutative PIDs implies that the generating rank of $P_2/ \left( \ker(\iota_1) \cap P_2\right)$  is $m \geq n=2g$.
\\

{Now we finish the proof that $h \geq g$ by showing that the generating rank of $\ker(\iota_2)/ (\ker(\iota_1) \cap \ker(\iota_2))$ is no more than $2h$.}
% i.e.\ that if $\Delta_1'$ and $\Delta_2'$ have a common stabilization $F$ of genus $h \leq g$ then in fact we must have $h=g$.
Let $P_F:= \ker(H_1^{\phi}(X_K) \to H_1^{\Phi}(X_F))$. By Proposition~\ref{prop:kerneloftwistedH1} applied to $\Delta_1'$ and $F$, we have that $P_F$ is generated as a $\Z[\xi]$-module by $\ker(H_1^{\phi}(X_K) \to H_1^{\Phi}(X_{\Delta_1'}))$ together with some $2h$ elements $x_1, \dots, x_{2h}$.  Here we use that the ring of Eisenstein integers $\Z[\xi]$ is a Euclidean domain and is therefore a PID.
However, by Proposition~\ref{prop:2knotsandtwistedh1} we have that
 \[\ker\big( H_1^{\phi}(X_K) \to H_1^{\Phi}(X_{\Delta_1'})\big)= \ker\big( H_1^{\phi}(X_K) \to H_1^{\Phi}(X_{\Delta_1})\big)= \ker(\iota_1).
\]
{So for any submodule $P$ of $P_F$, the quotient module $P/ (P \cap \ker(\iota_1))$ is isomorphic to a submodule of $P_F/  \ker(\iota_1)$ and hence, by Lemma~\ref{lemma:gen-rank-facts}~(\ref{item:gen-rank-subgroup}), has generating rank at most $2h$.}
But Proposition~\ref{prop:kerneloftwistedH1}  applied to $\Delta_2'$ and $F$ together with the fact that by Proposition~\ref{prop:2knotsandtwistedh1}
 \[\ker\big( H_1^{\phi}(X_K) \to H_1^{\Phi}(X_{\Delta_2'})\big)= \ker\big( H_1^{\phi}(X_K) \to H_1^{\Phi}(X_{\Delta_2})\big)= \ker(\iota_2)
\]
implies that $\ker(\iota_2)$ is contained in $P_F$.
We can therefore conclude as desired that
\[2h \geq \gr \left( \ker(\iota_2)/ \left( \ker(\iota_2) \cap \ker(\iota_1)\right)\right) \geq 2g. \qedhere.\]
\end{proof}

\bibliographystyle{amsalpha}
\def\MR#1{}
\bibliography{research}

\def\cprime{$'$}
\providecommand{\bysame}{\leavevmode\hbox to3em{\hrulefill}\thinspace}
\providecommand{\MR}{\relax\ifhmode\unskip\space\fi MR }
% \MRhref is called by the amsart/book/proc definition of \MR.
\providecommand{\MRhref}[2]{%
  \href{http://www.ams.org/mathscinet-getitem?mr=#1}{#2}
}
\providecommand{\href}[2]{#2}
\begin{thebibliography}{HKL10}

\bibitem[BP16]{BP}
Maciej Borodzik and Mark Powell, \emph{Embedded {M}orse {T}heory and {R}elative
  {S}plitting of {C}obordisms of {M}anifolds}, J. Geom. Anal. \textbf{26}
  (2016), no.~1, 57--87. \MR{3441503}

\bibitem[BS15]{Bay-Sun-15}
R.~İnanç Baykur and Nathan Sunukjian, \emph{Knotted surfaces in 4-manifolds
  and stabilizations}, Journal of Topology \textbf{9} (2015), no.~1, 215--231.

\bibitem[CG78]{Casson-Gordon:1978-1}
Andrew Casson and Cameron Gordon, \emph{On slice knots in dimension three},
  Algebraic and geometric topology (Proc. Sympos. Pure Math., Stanford Univ.,
  Stanford, Calif., 1976), Part 2, Amer. Math. Soc., Providence, R.I., 1978,
  pp.~39--53. \MR{81g:57003}

\bibitem[CG86]{Casson-Gordon:1986-1}
\bysame, \emph{Cobordism of classical knots}, \`A la recherche de la topologie
  perdue, Birkh\"auser Boston, Boston, MA, 1986, With an appendix by P. M.
  Gilmer, pp.~181--199. \MR{900 252}

\bibitem[CP19]{Conway-Powell}
Anthony Conway and Mark Powell, \emph{Enumerating homotopy ribbon slice discs},
  ar{X}iv:1902.05321, 2019.

\bibitem[DK01]{Davis-Kirk}
Jim Davis and Paul Kirk, \emph{Lecture notes in algebraic topology}, Graduate
  Studies in Mathematics, vol.~35, American Mathematical Society, Providence,
  RI, 2001. \MR{MR1841974 (2002f:55001)}

\bibitem[Fri04]{Friedl:2003-4}
Stefan Friedl, \emph{Eta invariants as sliceness obstructions and their
  relation to {C}asson-{G}ordon invariants}, Algebr. Geom. Topol. \textbf{4}
  (2004), 893--934 (electronic). \MR{MR2100685 (2005j:57016)}

\bibitem[GS99]{Gompf-stipsicz-book}
Robert~E. Gompf and Andr{\'a}s~I. Stipsicz, \emph{{$4$}-manifolds and {K}irby
  calculus}, Graduate Studies in Mathematics, vol.~20, American Mathematical
  Society, Providence, RI, 1999. \MR{1707327 (2000h:57038)}

\bibitem[HKL10]{HKL08}
Chris Herald, Paul Kirk, and Charles Livingston, \emph{Metabelian
  representations, twisted {A}lexander polynomials, knot slicing, and
  mutation}, Math. Z. \textbf{265} (2010), no.~4, 925--949. \MR{2652542
  (2011g:57006)}

\bibitem[JZ18a]{Juhasz-Zemke-slice-disks}
Andr\'{a}s Juh\'{a}sz and Ian Zemke, \emph{Distinguishing slice disks using
  knot {F}loer homology}, ar{X}iv:1804.09589, 2018.

\bibitem[JZ18b]{Juhasz-Zemke-stab-dist}
\bysame, \emph{Stabilization distance bounds from link {F}loer homology},
  ar{X}iv:1810.09158, 2018.

\bibitem[KL99]{Kirk-Livingston:1999-2}
Paul Kirk and Charles Livingston, \emph{Twisted {A}lexander invariants,
  {R}eidemeister torsion, and {C}asson-{G}ordon invariants}, Topology
  \textbf{38} (1999), no.~3, 635--661. \MR{2000c:57010}

\bibitem[KL01]{Kirk-Livingston:1999-1}
\bysame, \emph{Concordance and mutation}, Geom. Topol. \textbf{5} (2001),
  831--883 (electronic). \MR{MR1871406 (2002j:57016)}

\bibitem[KL05]{KimLivingston:2005}
Se-Goo Kim and Charles Livingston, \emph{Knot mutation: 4-genus of knots and
  algebraic concordance}, Pacific J. Math. \textbf{220} (2005), no.~1, 87--105.
  \MR{2195064}

\bibitem[Let00]{Let00}
Carl~F. Letsche, \emph{An obstruction to slicing knots using the eta
  invariant}, Math. Proc. Cambridge Philos. Soc. \textbf{128} (2000), no.~2,
  301--319. \MR{1735303 (2001b:57017)}

\bibitem[Lic97]{Lickorish-text}
W.~B.~Raymond Lickorish, \emph{An introduction to knot theory}, Graduate Texts
  in Mathematics, vol. 175, Springer-Verlag, New York, 1997. \MR{1472978}

\bibitem[Liv82]{livingston-surfaces}
Charles Livingston, \emph{Surfaces bounding the unlink}, Michigan Math. J.
  \textbf{29} (1982), no.~3, 289--298. \MR{674282}

\bibitem[Miy86]{Miyazaki-86}
Katura Miyazaki, \emph{On the relationship among unknotting number, knotting
  genus and {A}lexander invariant for {$2$}-knots}, Kobe J. Math. \textbf{3}
  (1986), no.~1, 77--85. \MR{867806}

\bibitem[Per75]{Pe}
Bernard Perron, \emph{Pseudo-isotopies de plongements en codimension {$2$}},
  Bull. Soc. Math. France \textbf{103} (1975), no.~3, 289--339. \MR{0394701 (52
  \#15500)}

\bibitem[Swe01]{Swenton}
Frank~J. Swenton, \emph{On a calculus for 2-knots and surfaces in 4-space}, J.
  Knot Theory Ramifications \textbf{10} (2001), no.~8, 1133--1141. \MR{1871221}

\bibitem[Wei94]{Weibel94}
Charles~A. Weibel, \emph{An introduction to homological algebra}, Cambridge
  Studies in Advanced Mathematics, vol.~38, Cambridge University Press,
  Cambridge, 1994. \MR{1269324 (95f:18001)}

\end{thebibliography}

\end{document}